\documentclass[reqno]{amsart}
\usepackage{amsmath, amsthm, amssymb, amstext}

\usepackage[left=2.9cm,right=2.9cm,top=3cm,bottom=3cm]{geometry}
\usepackage{hyperref,xcolor}
\hypersetup{pdfborder={0 0 0},colorlinks}
\usepackage{enumitem}
\allowdisplaybreaks
\usepackage{todonotes}

\newtheorem{theorem}{Theorem}
\newtheorem{remark}[theorem]{Remark}
\newtheorem{lemma}[theorem]{Lemma}
\newtheorem{proposition}[theorem]{Proposition}

\newtheorem{definition}[theorem]{Definition}

\let\div\undefined\DeclareMathOperator{\div}{div}
\newcommand{\Lp}[1]{L^{#1}(\Omega)}
\newcommand{\Wpzero}[1]{W^{1,#1}_0(\Omega)}
\newcommand*\diff{\mathrm{d}}
\newcommand\V{W^{1,\mathcal{H}}_0(\Omega)}
\newcommand{\N}{\mathbb{N}}
\newcommand{\R}{\mathbb{R}}

\numberwithin{theorem}{section}
\numberwithin{equation}{section}

\title[On double phase Kirchhoff problems with singular nonlinearity]
{On double phase Kirchhoff problems with singular nonlinearity}

\author[R.\,Arora]{Rakesh Arora}
\address[R.\,Arora]{Department of Mathematics and Statistics, Masaryk University, Building 08, Kotl\'{a}\v{r}sk\'{a} 2, Brno 611 37, Czech Republic}
\email{arora@math.muni.cz, arora.npde@gmail.com}

\author[A.\,Fiscella]{Alessio Fiscella}
\address[A.\,Fiscella]{Dipartimento di Matematica e Applicazioni, Universit\`a degli Studi di Milano-Bicocca, Via Cozzi 55, Milano, CAP 20125, Italy}
\email{alessio.fiscella@unimib.it}

\author[T.\,Mukherjee]{Tuhina Mukherjee}
\address[T.\,Mukherjee]{Department of Mathematics, Indian Institute of Technology Jodhpur, Rajasthan-506004, India-342037}
\email{tuhina@iitj.ac.in}

\author[P.\,Winkert]{Patrick Winkert}
\address[P.\,Winkert]{Technische Universit\"{a}t Berlin, Institut f\"{u}r Mathematik, Stra\ss e des 17.\,Juni 136, 10623 Berlin, Germany}
\email{winkert@math.tu-berlin.de}

\begin{document}

\begin{abstract}
	In this paper, we study multiplicity results for double phase problems of Kirchhoff type with right-hand sides that include a parametric singular term and a nonlinear term of subcritical growth. Under very general assumptions on the data, we prove the existence of at least two weak solutions that have different energy sign. Our treatment is based on the fibering method in form of the Nehari manifold. We point out that we cover both the non-degenerate as well as the degenerate Kirchhoff case in our setting.
\end{abstract}


\maketitle

\section{Introduction}

In this work, we are concerned with multiple solutions for double phase problems with a nonlocal Kirchhoff term and a singular right-hand side. To be more precise, we study the problem
\begin{equation}\label{problem}\tag{$P_\lambda$}
	\begin{aligned}
		-m \left[\int_\Omega \left( \frac{|\nabla u|^p}{p} + a(x) \frac{|\nabla u|^q}{q}\right)\,\diff x\right]\mathcal{L}_{p,q}^{a}(u) &= \lambda u^{-\gamma} +u^{r-1} \quad&& \text{in } \Omega,\\
		u  &> 0 \quad && \text{in } \Omega,\\
		u  &= 0 &&\text{on } \partial\Omega,
	\end{aligned}
\end{equation}
with $\Omega \subset \mathbb{R}^N$ ($N\geq 2$) being a bounded domain with Lipschitz boundary $\partial\Omega$, $\lambda>0$ is the parameter to be specified and $\mathcal{L}_{p,q}^{a}$ denotes the double phase operator given by
\begin{align}\label{operator_double_phase}
	\mathcal{L}_{p,q}^{a}(u):= \div \left(|\nabla u|^{p-2}\nabla u + a(x) |\nabla u|^{q-2}\nabla u \right), \quad u\in\V.
\end{align}
Furthermore, we suppose the following conditions:
\begin{enumerate}
	\item[\textnormal{(H)}]
		\begin{enumerate}
			\item[\textnormal{(i)}]
				$1<p<N$, $p<q<p^*$ and $0 \leq a(\cdot)\in L^\infty(\Omega)$ with $p^*$ being the critical Sobolev exponent to $p$ given by
				\begin{align}\label{critical_exponent}
					p^*=\frac{Np}{N-p};
				\end{align}
			\item[\textnormal{(ii)}]
				$0<\gamma<1$ and $m\colon  [0,\infty) \to [0,\infty)$ is a continuous function defined by
				\begin{align}\label{kirchhoff}
					m(t) = a_0 + b_0 t^{\theta-1} \quad \text{for all } t\geq 0,
				\end{align} 
				where $a_0\geq 0$, $b_0>0$ with $ \theta \in \left[1,\frac{r}{q}\right)$ and $r\in (q\theta,p^*)$.
		\end{enumerate}
\end{enumerate}

Problems of type \eqref{problem} combine several interesting phenomena into one problem. First, the differential operator involved is the so-called double phase operator given in \eqref{operator_double_phase}. In 1986, Zhikov \cite{Zhikov-1986} introduced for the first time in literature the related energy functional to \eqref{operator_double_phase} defined by
\begin{align}\label{integral_minimizer}
	\omega \mapsto \int_\Omega \big(|\nabla  \omega|^p+a(x)|\nabla  \omega|^q\big)\,\diff x.
\end{align}
This kind of functional has been used to describe models for strongly anisotropic materials in the context of homogenization and elasticity.
Indeed, the hardening properties of strongly anisotropic materials change point by point. For this, the modulating coefficient $a(\cdot)$ helps to regulate the mixture of two different materials, with hardening powers $p$ and $q$. 
From the mathematical point of view, the behavior of \eqref{integral_minimizer} is related to the sets on which the weight function $a(\cdot)$ vanishes or not. Hence, there are two phases $a(x)=0$ and $a(x)\neq 0$ and so \eqref{integral_minimizer} is said to be of double phase type. In this direction, functional \eqref{integral_minimizer} has several mathematical applications in the study of duality theory and of the Lavrentiev
gap phenomenon, see Zhikov \cite{Zhikov-1995,Zhikov-2011}. Also, \eqref{integral_minimizer} belongs to the class of the integral functionals
with nonstandard growth condition, according to Marcellini’s terminology \cite{Marcellini-1991,Marcellini-1989b}.
Following this line of research, Mingione et al.\,provide famous results in the regularity theory of local minimizers of \eqref{integral_minimizer}, see, for example, the works in Baroni-Colombo-Mingione \cite{Baroni-Colombo-Mingione-2015, Baroni-Colombo-Mingione-2018} and Colombo-Mingione \cite{Colombo-Mingione-2015a, Colombo-Mingione-2015b}.

A second interesting phenomenon is the appearance of a nonlocal Kirchhoff term given in \eqref{kirchhoff} which was first introduced by Kirchhoff \cite{Kirchhoff-1876}. Problems as in \eqref{problem} involving a Kirchhoff term are said to be degenerate if $a_0=0$ and nondegenerate if $a_0>0$. It is worth noting that the degenerate case is rather interesting and is treated in well-known papers in the Kirchhoff theory. We do cover the degenerate case in our paper which has several applications in physics. For example, the transverse oscillations of a stretched string with nonlocal flexural rigidity depends continuously on the Sobolev deflection norm of $u$ via $m(\int_\Omega |\nabla u|^2\,\diff x)$, that is, $m(0)=0$ is nothing less than the base tension of the string is zero. In the extensive literature on degenerate and nondegenerate Kirchhoff problems, we refer, for example, to the works of Arora-Giacomoni-Mukherjee-Sreenadh \cite{Arora-Giacomoni-Mukherjee-Sreenadh-2019,Arora-Giacomoni-Mukherjee-Sreenadh-2020}, Autuori-Pucci-Salvatori \cite{Autuori-Pucci-Salvatori-2010}, D'Ancona-Spagnolo \cite{DAncona-Spagnolo-1992}, Figueiredo \cite{Figueiredo-2013}, Fiscella \cite{Fiscella-2019}, Fiscella-Valdinoci \cite{Fiscella-Valdinoci-2014}, Mingqi-R\u{a}dulescu-Zhang \cite{Mingqi-Radulescu-Zhang-2019}, Pucci-Xiang-Zhang \cite{Pucci-Xiang-Zhang-2015}, Xiang-Zhang-R\u{a}dulescu \cite{Xiang-Zhang-Radulescu-2016} and the references therein.

A third fascinating aspect of our problem is the presence of a nonlinear singular term in \eqref{problem}. The study of elliptic or integral equations involving singular terms started in the early sixties by the work of Fulks-Maybee \cite{Fulks-Maybee-1960}, originating from the models of heat conduction in electrically conducting materials. More precisely, let $\Omega$ be an electrically conducting medium in $\mathbb{R}^3$ and $u$ be the steady state temperature distribution in the region $\Omega$. Then, if $\frac{\lambda}{u^{\gamma}}$ is the rate of generation of heat with constant voltage $\lambda$ (as in our model $(P_\lambda)$), then the temperature distribution in the conducting medium satisfies the local and linear counterpart of the equation mentioned in \eqref{problem}. For interested readers, we refer to works of Diaz-Morel-Oswald \cite{Diaz-Morel-Oswald-1987}, Nachman-Callegari \cite{Nachman.al_1986}, and Stuart \cite{Stuart_1974} for applications in non-newtonian fluid flows in porous media and heterogeneous catalysts, pseudo-plastic fluids, and in the theory of radiative transfer in semi-infinite atmospheres respectively.

Denoting
\begin{align*}
	\phi_{\mathcal{H}}(u)= \int_\Omega \left(\frac{|u|^p}{p}+a(x) \frac{|u|^q}{q}\right)\,\diff x,
\end{align*}
and indicating with $\V$ the homogeneous Musielak-Orlicz Sobolev space which will be introduced in Section \ref{sec_2},
we can state the following definition of a weak solution to problem \eqref{problem}.

\begin{definition}
	A function $u \in \V$ is said to be a weak solution of problem \eqref{problem} if $u^{-\gamma}\varphi\in L^1(\Omega)$, $u>0$ a.e.\,in $\Omega$ and
	\begin{align*}
		m(\phi_\mathcal{H}(\nabla u)) \left\langle \mathcal{L}_{p,q}^{a}(u), \varphi\right\rangle= \lambda \int_{\Omega} u^{-\gamma}\varphi\,\diff x +\int_\Omega u^{r-1} \varphi\,\diff x
	\end{align*}
	is satisfied for all $\varphi \in \V$, where $\langle\cdot,\cdot\rangle$ denotes the duality pairing between $\V$ and its dual space $\V^*$.
\end{definition}

Based on hypotheses \textnormal{(H)} and Proposition \ref{proposition_modular_properties} in Section \ref{sec_2}, it is clear that the definition of a weak solution is well-defined. Moreover, we introduce the corresponding energy functional $J_\lambda\colon \V \to \R$ associated to problem \eqref{problem} defined by
\begin{align*}
	J_\lambda (u) = M[\phi_{\mathcal{H}}(\nabla u)] -\frac{\lambda}{1-\gamma}\int_\Omega |u|^{1-\gamma}\,\diff x - \frac{1}{r}\int_\Omega |u|^r\,\diff x,
\end{align*}
where $M\colon  [0,\infty) \to [0,\infty)$ is given by
\begin{align*}
	M(t)= \int_0^tm(\tau)\,\diff \tau= a_0 {t} +\frac{b_0}{\theta}t^\theta.
\end{align*}

The main result in this paper reads as follows.

\begin{theorem}\label{main_result}
	Let hypotheses \textnormal{(H)}  be satisfied. Then there exists $\lambda^*>0$ such that for all $\lambda \in (0,\lambda^*]$ problem \eqref{problem} has at least two weak solutions $u_\lambda$, $v_\lambda \in \V$ such that $J_\lambda(u_\lambda)<0<J_\lambda(v_\lambda)$.
\end{theorem}

The proof of Theorem \ref{main_result} is based on a careful study of the corresponding fibering map which was initiated by the work of Dr\'{a}bek-Pohozaev \cite{Drabek-Pohozaev-1997}. The idea is to define the related Nehari manifold of \eqref{problem} in the form of the first derivative of the fibering mapping and then we split the Nehari manifold into three disjoint parts which are related to the second derivative of the fibering function. It turns out that the global minimizers of $J_\lambda$ restricted to two of them are the solutions we seek and the third one is the empty set for small values of the parameter $\lambda>0$. This method has become a very powerful tool and has been used, for example, in the works of Chen-Kuo-Wu \cite{Chen-Kuo-Wu-2011} (for degenerate Kirchhoff Laplacian problems with sign-changing weight), Fiscella-Mishra \cite{Fiscella-Mishra-2019} (for fractional Kirchhoff problems), Kumar-R\u{a}dulescu-Sreenadh \cite{Kumar-Radulescu-Sreenadh-2020} (for critical $(p,q)$-equations), Liao-Zhang-Liu-Tang \cite{Liao-Zhang-Liu-Tang-2015} (for Kirchhoff Laplacian problems),   Mukherjee-Sreenadh \cite{Mukherjee-Sreenadh-2019} (for fractional problems), Papageorgiou-Repov\v{s}-Vetro \cite{Papageorgiou-Repovs-Vetro-2021} (for $(p,q)$-equations), Papageorgiou-Winkert \cite{Papageorgiou-Winkert-2021} (for $p$-Laplacian problems), Tang-Chen \cite{Tang-Chen-2017} (for ground state solutions of Schr\"odinger type), see also the references therein.

To the best of our knowledge, the only work dealing with a double phase operator and a nonlocal Kirchhoff term has been recently done by Fiscella-Pinamonti \cite{Fiscella-Pinamonti-2020} who studied the problem
\begin{equation}\label{problem2}
	\begin{aligned}
		-m \left[\int_\Omega \left( \frac{|\nabla u|^p}{p} + a(x) \frac{|\nabla u|^q}{q}\right)\,\diff x\right]\mathcal{L}_{p,q}^{a}(u) &=f(x,u) \quad&& \text{in } \Omega,\\
		u  &= 0 &&\text{on } \partial\Omega,
	\end{aligned}
\end{equation}
with a Carath\'eodory $f\colon\Omega\times\R\to\R$ satisfying subcritical growth and the Ambrosetti-Rabinowitz condition. Based on the mountain-pass theorem, the existence of a nontrivial weak solution of \eqref{problem2} is shown. In addition, the authors in \cite{Fiscella-Pinamonti-2020} considered the problem 
\begin{equation}\label{problem3}
	\begin{aligned}
		-m \left( \int_\Omega |\nabla u|^p\,\diff x\right)\Delta_pu-m\left(\int_\Omega a(x)|\nabla u|^q\,\diff x\right) \div\left(a(x)|\nabla u|^{q-2}\nabla u\right)
		&=f(x,u) \quad&& \text{in } \Omega,\\
		u  &= 0 &&\text{on } \partial\Omega,
	\end{aligned}
\end{equation}
and proved the existence of infinitely many weak solutions with unbounded energy by using the fountain theorem. Even if the double phase operator does not explicitly appear in \eqref{problem3}, this problem has still a variational structure set in the same double phase framework of \eqref{problem2}. 


Finally, we mention some existence and multiplicity results for double phase problems without Kirchhoff term, that is, $m(t)\equiv 1$ for all $t\geq 0$. We refer to the papers of Arora-Shmarev \cite{Arora-Shmarev-2021,Arora-Shmarev-2020} (parabolic double phase problems), Colasuonno-Squassina \cite{Colasuonno-Squassina-2016} (eigenvalue problems), Farkas-Winkert \cite{Farkas-Winkert-2021}, Farkas-Fiscella-Winkert \cite{Farkas-Fiscella-Winkert-2021} (singular Finsler double phase problems), Fiscella \cite{Fiscella-2020} (Hardy potentials), Gasi\'nski-Papa\-georgiou \cite{Gasinski-Papageorgiou-2019} (locally Lipschitz right-hand side), Gasi\'nski-Winkert \cite{Gasinski-Winkert-2020a,Gasinski-Winkert-2020b,Gasinski-Winkert-2021} (convection and superlinear problems), Liu-Dai \cite{Liu-Dai-2018} (Nehari manifold approach), Liu-Dai-Papageorgiou-Winkert \cite{Liu-Dai-Papageorgiou-Winkert-2021} (singular problems),
Perera-Squassina \cite{Perera-Squassina-2018} (Morse theoretical approach), Zeng-Bai-Gasi\'nski-Winkert \cite{Zeng-Bai-Gasinski-Winkert-2020, Zeng-Gasinski-Winkert-Bai-2020} (multivalued obstacle problems) and the references therein.

The paper is organized as follows. In Section \ref{sec_2}, we recall the main properties of Musielak-Orlicz Sobolev spaces $\V$ and state the main embeddings concerning these spaces. Section \ref{sec_3} gives a detailed analysis of the fibering map and presents the main properties of the three disjoints subsets of the Nehari manifold. In Section \ref{sec_4} we prove the existence of at least two weak solutions of problem \eqref{problem}, see Propositions \ref{existence:first} and \ref{existence:second}. Finally, in Section \ref{sec final} we study a singular Kirchhoff problem driven by the left-hand side of \eqref{problem3}, inspired by \cite{Fiscella-Pinamonti-2020}.

\section{Preliminaries}\label{sec_2}

In this section, we will recall the main properties and embedding results for Musielak-Orlicz Sobolev spaces. To this end, we suppose that $\Omega\subset \R^N$ ($N\geq 2$) is a bounded domain with Lipschitz boundary $\partial\Omega$. For any $r\in [1,\infty)$, we denote by $\Lp{r}=L^r(\Omega;\R)$ and $L^r(\Omega;\R^N)$ the usual Lebesgue spaces with the norm $\|\cdot\|_r$. Moreover, the Sobolev space $\Wpzero{r}$ is equipped with the equivalent norm $\|\nabla \cdot \|_r$ for $1<r<\infty$.

Let hypothesis \textnormal{(H)(i)} be satisfied and consider the nonlinear function $\mathcal{H}\colon \Omega \times [0,\infty)\to [0,\infty)$ defined by
\begin{align*}
	\mathcal H(x,t)= t^p+a(x)t^q.
\end{align*}
Denoting by $M(\Omega)$ the space of all measurable functions $u\colon\Omega\to\R$, we can introduce the Musielak-Orlicz Lebesgue space $L^\mathcal{H}(\Omega)$ which is given by
\begin{align*}
	L^\mathcal{H}(\Omega)=\left \{u\in M(\Omega)\,:\,\varrho_{\mathcal{H}}(u)<\infty \right \}
\end{align*}
equipped with the Luxemburg norm
\begin{align*}
	\|u\|_{\mathcal{H}} = \inf \left \{ \tau >0\,:\, \varrho_{\mathcal{H}}\left(\frac{u}{\tau}\right) \leq 1  \right \},
\end{align*}
where the modular function is given by
\begin{align*}
	\varrho_{\mathcal{H}}(u):=\int_\Omega \mathcal{H}(x,|u|)\,\diff x=\int_\Omega \big(|u|^{p}+a(x)|u|^q\big)\,\diff x.
\end{align*}

The norm $\|\,\cdot\,\|_{\mathcal{H}}$ and the modular function $\varrho_{\mathcal{H}}$ have the following relations, see Liu-Dai \cite[Proposition 2.1]{Liu-Dai-2018} or Crespo-Blanco-Gasi\'nski-Harjulehto-Winkert \cite[Proposition 2.13]{Crespo-Blanco-Gasinski-Harjulehto-Winkert-2021}.

\begin{proposition}\label{proposition_modular_properties}
	Let \textnormal{(H)(i)} be satisfied, $u\in L^{\mathcal{H}}(\Omega)$ and $c>0$. Then the following hold:
	\begin{enumerate}
		\item[\textnormal{(i)}]
			If $u\neq 0$, then $\|u\|_{\mathcal{H}}=c$ if and only if $ \varrho_{\mathcal{H}}(\frac{u}{c})=1$;
		\item[\textnormal{(ii)}]
			$\|u\|_{\mathcal{H}}<1$ (resp.\,$>1$, $=1$) if and only if $ \varrho_{\mathcal{H}}(u)<1$ (resp.\,$>1$, $=1$);
		\item[\textnormal{(iii)}]
			If $\|u\|_{\mathcal{H}}<1$, then $\|u\|_{\mathcal{H}}^q\leq \varrho_{\mathcal{H}}(u)\leq\|u\|_{\mathcal{H}}^p$;
		\item[\textnormal{(iv)}]
			If $\|u\|_{\mathcal{H}}>1$, then $\|u\|_{\mathcal{H}}^p\leq \varrho_{\mathcal{H}}(u)\leq\|u\|_{\mathcal{H}}^q$;
		\item[\textnormal{(v)}]
			$\|u\|_{\mathcal{H}}\to 0$ if and only if $ \varrho_{\mathcal{H}}(u)\to 0$;
		\item[\textnormal{(vi)}]
			$\|u\|_{\mathcal{H}}\to \infty$ if and only if $ \varrho_{\mathcal{H}}(u)\to \infty$.
	\end{enumerate}
\end{proposition}

Furthermore, we define the seminormed space
\begin{align*}
	L^q_a(\Omega)=\left \{u\in M(\Omega)\,:\,\int_\Omega a(x) |u|^q \,\diff x< \infty \right \}
\end{align*}
endowed with the seminorm
\begin{align*}
	\|u\|_{q,a} = \left(\int_\Omega a(x) |u|^q \,\diff x \right)^{\frac{1}{q}}.
\end{align*}
While, the corresponding Musielak-Orlicz Sobolev space $W^{1,\mathcal{H}}(\Omega)$ is defined by
\begin{align*}
	W^{1,\mathcal{H}}(\Omega)= \Big \{u \in L^\mathcal{H}(\Omega) \,:\, |\nabla u| \in L^{\mathcal{H}}(\Omega) \Big\}
\end{align*}
equipped with the norm
\begin{align*}
	\|u\|_{1,\mathcal{H}}= \|\nabla u \|_{\mathcal{H}}+\|u\|_{\mathcal{H}},
\end{align*}
where $\|\nabla u\|_\mathcal{H}=\|\,|\nabla u|\,\|_{\mathcal{H}}$. Moreover, we denote by $W^{1,\mathcal{H}}_0(\Omega)$ the completion of $C^\infty_0(\Omega)$ in $W^{1,\mathcal{H}}(\Omega)$. From hypothesis \textnormal{(H)(i)}, we know that we can equip the space $\V$ with the equivalent norm given by
\begin{align*}
	\|u\|=\|\nabla u\|_{\mathcal{H}},
\end{align*}
see Proposition  2.16(ii) of Crespo-Blanco-Gasi\'nski-Harjulehto-Winkert \cite{Crespo-Blanco-Gasinski-Harjulehto-Winkert-2021}. It is known that $L^\mathcal{H}(\Omega)$, $W^{1,\mathcal{H}}(\Omega)$ and $W^{1,\mathcal{H}}_0(\Omega)$ are uniformly convex and so reflexive Banach spaces, see Colasuonno-Squassina \cite[Proposition 2.14]{Colasuonno-Squassina-2016} or Harjulehto-H\"{a}st\"{o} \cite[Theorem 6.1.4]{Harjulehto-Hasto-2019}.

We end this section by recalling the following embeddings for the spaces $L^\mathcal{H}(\Omega)$ and $W^{1,\mathcal{H}}_0(\Omega)$, see Colasuonno-Squassina \cite[Proposition 2.15]{Colasuonno-Squassina-2016} or Crespo-Blanco-Gasi\'nski-Harjulehto-Winkert \cite[Propositions 2.17 and 2.19]{Crespo-Blanco-Gasinski-Harjulehto-Winkert-2021}.

\begin{proposition}\label{proposition_embeddings}
	Let \textnormal{(H)(i)} be satisfied and let $p^*$ be the critical exponent to $p$ given in \eqref{critical_exponent}. Then the following embeddings hold:
	\begin{enumerate}
		\item[\textnormal{(i)}]
			$\Lp{\mathcal{H}} \hookrightarrow \Lp{r}$ and $\V\hookrightarrow \Wpzero{r}$ are continuous for all $r\in [1,p]$;
		\item[\textnormal{(ii)}]
			$\V \hookrightarrow \Lp{r}$ is continuous for all $r \in [1,p^*]$ and compact for all $r \in [1,p^*)$;
		\item[\textnormal{(iii)}]
			$\Lp{\mathcal{H}} \hookrightarrow L^q_a(\Omega)$ is continuous;
		\item[\textnormal{(iv)}]
			$\Lp{q}\hookrightarrow\Lp{\mathcal{H}} $ is continuous.
	\end{enumerate}
\end{proposition}

\begin{remark}
	Since $q<p^*$ by hypothesis \textnormal{(H)(i)}, we know from Proposition \ref{proposition_embeddings}\textnormal{(ii)} that $\V \hookrightarrow L^q(\Omega)$ is compact.
\end{remark}

\section{Analysis of the fibering function}\label{sec_3}

As mentioned in the Introduction, the proof of Theorem \ref{main_result} relies on the fibering map corresponding to our problem. For this purpose, we recall that the energy functional $J_\lambda\colon \V \to \R$ related to problem \eqref{problem} is given by
\begin{align*}
	J_\lambda (u) 
	&= M[\phi_{\mathcal{H}}(\nabla u)] -\frac{\lambda}{1-\gamma}\int_\Omega |u|^{1-\gamma}\,\diff x - \frac{1}{r}\int_\Omega |u|^r\,\diff x\\
	&=\left[a_0\phi_\mathcal{H}(\nabla u) 
	+\frac{b_0}{\theta}\phi_\mathcal{H}^{\theta}(\nabla u)\right]
	- \frac{\lambda}{1-\gamma}\int_\Omega |u|^{1-\gamma}\,\diff x
	- \frac{1}{r}\int_\Omega |u|^r\,\diff x.
\end{align*}
Due to the presence of the singular term, we know that $J_\lambda$ is not $C^1$. For $u \in \V \setminus \{0\}$, we introduce the fibering function $\psi_u\colon [0,\infty) \to \R$ defined by 
\begin{align*}
	\psi_u(t)= J_\lambda(tu)\quad \text{for all }t\geq 0,
\end{align*}
which gives
\begin{align*}
	\psi_u(t)
	= \left[a_0\phi_\mathcal{H}(t\nabla u) 
	+\frac{b_0}{\theta}\phi_\mathcal{H}^{\theta}(t\nabla u)\right]
	- \lambda \frac{t^{1-\gamma}}{1-\gamma}\int_\Omega |u|^{1-\gamma}\,\diff x
	- \frac{t^r}{r}\int_\Omega |u|^r\,\diff x.
\end{align*}
Note that $\psi_u \in C^{\infty}((0,\infty))$. In particular, we have for $t>0$
\begin{align*}
	\psi_u'(t) 
	&= \left[a_0 +b_0\phi_\mathcal{H}^{\theta-1}(t\nabla u)\right] \left(t^{p-1}\|\nabla u\|_p^p+t^{q-1} \|\nabla u\|_{q,a}^q\right)
	- \lambda t^{-\gamma}\int_\Omega |u|^{1-\gamma}\,\diff x 
	- t^{r-1}\int_\Omega |u|^r\,\diff x
\end{align*}
and
\begin{align}\label{second-derivative}
	\psi_u''(t)
	& =   \left[a_0 +b_0\phi_\mathcal{H}^{\theta-1}(t\nabla u)\right] \left[(p-1)t^{p-2}\|\nabla u\|_p^p+(q-1)t^{q-2} \|\nabla u\|_{q,a}^q\right]\nonumber\\
	& \quad + b_0(\theta -1)\phi_\mathcal{H}^{\theta-2}(t\nabla u)\left(t^{p-1}\|\nabla u\|_p^p+t^{q-1} \|\nabla u\|_{q,a}^q\right)^2\\
	&\quad  + \lambda \gamma t^{-\gamma-1}\int_\Omega |u|^{1-\gamma}\,\diff x 
	-(r-1) t^{r-2}\int_\Omega |u|^r\,\diff x.\nonumber
\end{align}
Based on this we can introduce the Nehari manifold related to our problem which is defined by
\begin{align*}
	\mathcal{N}_\lambda = \left\{u\in \V\setminus\{0\}\,:\, \psi_u'(1)=0\right\}.
\end{align*}
Therefore, we have $u \in \mathcal{N}_\lambda $ if and only if
\begin{align*}
	\left[a_0 +b_0\phi_\mathcal{H}^{\theta-1}(\nabla u)\right] \left(\|\nabla u\|_p^p+ \|\nabla u\|_{q,a}^q\right)=  \lambda \int_\Omega |u|^{1-\gamma}\,\diff x + \int_\Omega |u|^r\,\diff x.
\end{align*}
Also $tu\in \mathcal{N}_\lambda $ if and only if $\psi_{tu}'(1)=0$. It is clear that $\mathcal{N}_\lambda $ contains all weak solutions of \eqref{problem} but it is smaller than the whole space $\V$. For our further study, we need to decompose the set $\mathcal{N}_\lambda $ in the following disjoints sets
\begin{align*}
	\mathcal{N}_\lambda^{+} 
	&= \left\{u\in \mathcal{N}_\lambda \,:\, \psi_u''(1)> 0\right\},\\
	\mathcal{N}_\lambda^{-} 
	&= \left\{u\in \mathcal{N}_\lambda \,:\, \psi_u''(1)< 0\right\},\\
	\mathcal{N}_\lambda^{\circ} 
	&= \left\{u\in \mathcal{N}_\lambda \,:\, \psi_u''(1)= 0\right\}.
\end{align*}

Now, we show the coercivity of the energy functional $J_\lambda$ restricted to the Nehari manifold $\mathcal{N}_\lambda$.

\begin{lemma}\label{ksdp-lem1}
	Let hypotheses \textnormal{(H)} be satisfied. Then $J_\lambda \big|_{\mathcal{N}_\lambda }$ is coercive and bounded from below for any $\lambda>0$.
\end{lemma}

\begin{proof}
	Let $u\in \mathcal{N}_\lambda $ be such that $\|u\|>1$. The definition of $\mathcal{N}_\lambda$ implies that
	\begin{align}\label{coercivity-1}
		-\frac{1}{r}\int_\Omega |u|^r\,\diff x
		=-\frac{1}{r}\left[a_0 +b_0\phi_\mathcal{H}^{\theta-1}(\nabla u)\right] \left(\|\nabla u\|_p^p+ \|\nabla u\|_{q,a}^q\right)+\frac{\lambda}{r}  \int_\Omega |u|^{1-\gamma}\,\diff x.
	\end{align}
	Using \eqref{coercivity-1}, $1-\gamma<1<p<q\leq\theta q<r$ along with Proposition \ref{proposition_modular_properties}(iv) and Proposition \ref{proposition_embeddings}(ii), we get 
	\begin{align*}
		J_\lambda(u) 
		& =\left[a_0\phi_\mathcal{H}(\nabla u)
		+\frac{b_0}{\theta}\phi_\mathcal{H}^{\theta}(\nabla u)\right]
		- \frac{\lambda}{1-\gamma}\int_\Omega |u|^{1-\gamma}\,\diff x
		- \frac{1}{r}\int_\Omega |u|^r\,\diff x\\
		& = a_0 \left(\phi_\mathcal{H}(\nabla u)-\frac{1}{r}\varrho_\mathcal{H}(\nabla u)\right) + b_0\phi_\mathcal{H}^{\theta-1}(\nabla u) \left(\frac{1}{\theta}\phi_\mathcal{H}(\nabla u)-\frac{1}{r}\varrho_\mathcal{H}(\nabla u)\right)\\
		&\quad +\lambda \left(\frac{1}{r}-\frac{1}{1-\gamma}\right)\int_\Omega |u|^{1-\gamma}\,\diff x\\
		&\geq  a_0\left[ \left(\frac{1}{p}-\frac{1}{r}\right) \|\nabla u\|_p^p + \left(\frac{1}{q}-\frac{1}{r}\right) \|\nabla u\|_{q,a}^q\right]\\
		&\quad+ b_0\phi_\mathcal{H}^{\theta-1}(\nabla u) \left[ \left(\frac{1}{p\theta}-\frac{1}{r}\right) \|\nabla u\|_p^p + \left(\frac{1}{q\theta}-\frac{1}{r}\right) \|\nabla u\|_{q,a}^q\right]
		-C_1 \|u\|^{1-\gamma}\\
		& \geq  b_0\phi_\mathcal{H}^{\theta-1}(\nabla u)\left(\frac{1}{q\theta}-\frac{1}{r}\right) \varrho_\mathcal{H}(\nabla u)
		-C_1\int_\Omega |u|^{1-\gamma}\,\diff x\\
		& \geq C_2 \|u\|^{p\theta}-C_1 \|u\|^{1-\gamma},
	\end{align*}
	with positive constants $C_1$ and $C_2$, where we have used the estimate
	\begin{align*}
		\phi_\mathcal{H}^{\theta-1}(\nabla u)\varrho_{\mathcal{H}}(\nabla u)
		=\left[\frac{1}{p}\|\nabla u\|_p^p+\frac{1}{q}\|\nabla u \|^q_{q,a}\right]^{\theta-1}\left[\|\nabla u\|_p^p+\|\nabla u\|^q_{q,a}\right]
		\geq \frac{1}{q^{\theta-1}}\varrho_{\mathcal{H}}^{\theta}(\nabla u) \geq \frac{1}{q^{\theta-1}}\|u\|^{p\theta}.
	\end{align*}
	Since $ p\theta\geq p>1-\gamma$, the coercivity of $J_\lambda\big|_{\mathcal{N}_\lambda }$ follows.  If we set
	\begin{align*}
		h(t) =C_1 t^{p\theta}- C_2t^{1-\gamma}\quad \text{for all }t>0,
	\end{align*}
	then it is easy to see that $h$ attains its unique minimum at
	\begin{align*}
		t_0 = \left(\frac{C_2(1-\gamma)}{C_1p\theta}\right)^{\frac{1}{p\theta-1+\gamma}}.
	\end{align*}
	Hence, $J_\lambda\big|_{\mathcal{N}_\lambda }$ is bounded from below. This completes the proof.
\end{proof}

Let $S$ be the best Sobolev constant in $W_0^{1,p}(\Omega)$ defined as 
\begin{align}\label{best-sobolev-constant}
	S=  \inf_{{ u\in W^{1,p}_0(\Omega)\setminus \{0\}}}\frac{\|\nabla u\|^p_p}{\|u\|_{p^*}^p}.
\end{align}

The next result shows the emptiness of $\mathcal{N}_\lambda^\circ$ for small values of $\lambda$.

\begin{lemma}\label{ksdp-lem3-}
	Let hypotheses \textnormal{(H)} be satisfied. Then there exists $\Lambda_1>0$ such that $\mathcal{N}_\lambda^\circ=\emptyset$ for all $\lambda \in (0,\Lambda_1)$.
\end{lemma}

\begin{proof}
	Arguing by contradiction, we suppose that for each $\lambda>0$, there exists  $u\in \V\setminus\{0\}$ such that $\psi_u'(1)=0=\psi_u''(1)$. That is
	\begin{align}\label{ksdp-11}
		\left[a_0 +b_0\phi_\mathcal{H}^{\theta-1}(\nabla u)\right] \left(\|\nabla u\|_p^p+  \|\nabla u\|_{q,a}^q\right)=\lambda \int_\Omega |u|^{1-\gamma}\,\diff x + \int_\Omega |u|^r\,\diff x
	\end{align}
	and
	\begin{align} \label{ksdp-12}
		\begin{split}
			& \left[a_0 +b_0\phi_\mathcal{H}^{\theta-1}(\nabla u)\right] \left[(p-1)\|\nabla u\|_p^p+(q-1) \|\nabla u\|_{q,a}^q\right]\\
			&\quad+ b_0(\theta -1)\phi_\mathcal{H}^{\theta-2}(\nabla u)\left(\|\nabla u\|_p^p+ \|\nabla u\|_{q,a}^q\right)^2  = -\lambda \gamma \int_\Omega |u|^{1-\gamma}\,\diff x +(r-1) \int_\Omega |u|^r\,\diff x.
		\end{split}
	\end{align}
	Multiplying \eqref{ksdp-11} with $\gamma$ and adding it to \eqref{ksdp-12} yields
	\begin{align}\label{ksdp-13}
		\begin{split}
			&\left[a_0 +b_0\phi_\mathcal{H}^{\theta-1}(\nabla u)\right] \left[(p-1+\gamma)\|\nabla u\|_p^p+(q-1+\gamma) \|\nabla u\|_{q,a}^q\right]\\
			&\quad+ b_0(\theta -1)\phi_\mathcal{H}^{\theta-2}(\nabla u)\left(\|\nabla u\|_p^p+ \|\nabla u\|_{q,a}^q\right)^2  = (r-1+\gamma) \int_\Omega |u|^r\,\diff x.
		\end{split}
	\end{align}
	On the other hand, subtracting \eqref{ksdp-12} from \eqref{ksdp-11} multiplied by $(r-1)$,  we obtain
	\begin{align}\label{ksdp-14}
		\begin{split}
			&\left[a_0 +b_0\phi_\mathcal{H}^{\theta-1}(\nabla u)\right] \left[(r-p)\|\nabla u\|_p^p+(r-q) \|\nabla u\|_{q,a}^q\right]\\
			&\quad- b_0(\theta -1)\phi_\mathcal{H}^{\theta-2}(\nabla u)\left(\|\nabla u\|_p^p+ \|\nabla u\|_{q,a}^q\right)^2  =
			(r-1+\gamma)\lambda \int_\Omega |u|^{1-\gamma}\,\diff x.
		\end{split}
	\end{align}
	We define the functional $T_\lambda \colon\mathcal{N}_\lambda \to \R$ given by
	\begin{align*}
		T_\lambda(u) &= \frac{\left[a_0 +b_0\phi_\mathcal{H}^{\theta-1}(\nabla u)\right] \left[(p-1+\gamma)\|\nabla u\|_p^p+(q-1+\gamma) \|\nabla u\|_{q,a}^q\right]}{ (r-1+\gamma)}\\
		&\quad +\frac{ b_0(\theta -1)\phi_\mathcal{H}^{\theta-2}(\nabla u)\left(\|\nabla u\|_p^p+ \|\nabla u\|_{q,a}^q\right)^2}{ (r-1+\gamma)}- \int_\Omega |u|^{r}\,\diff x.
	\end{align*}
	From \eqref{ksdp-13} we see that $T_\lambda (u)=0$ for all $u\in \mathcal{N}_\lambda^\circ$. Since $\theta \geq 1$, using H\"older's inequality and \eqref{best-sobolev-constant} along with the estimate $p \phi_{\mathcal{H}}(\nabla u) \geq \|\nabla u\|_p^p$ we obtain
	\begin{align}\label{estimate:1}
		\begin{split}
			T_\lambda (u) 
			&\geq  \frac{(p-1+\gamma)}{(r-1+\gamma)} \left(a_0\|\nabla u\|_p^p+ \frac{b_0}{p^{\theta-1}}\|\nabla u\|^{p\theta}_p\right) - S^{-\frac{r}{p}}|\Omega|^{1-\frac{r}{p^*}}\|\nabla u\|^r_p \\
			& \geq \frac{(p-1+\gamma)}{(r-1+\gamma)} \left( \frac{b_0}{p^{\theta-1}}\|\nabla u\|^{p\theta}_p\right) - S^{-\frac{r}{p}}|\Omega|^{1-\frac{r}{p^*}}\|\nabla u\|^r_p\\
			& = \|\nabla u\|^r_p \left( A \|\nabla u\|^{p\theta-r}_p -B\right),
		\end{split}
	\end{align}
	where
	\begin{align*}
		A:=  \left(\frac{(p-1+\gamma)b_0}{p^{\theta-1}(r-1+\gamma)}\right)>0
		\quad \text{and}\quad
		B:= S^{-\frac{r}{p}}|\Omega|^{1-\frac{r}{p^*}}>0.
	\end{align*}

	Since $p<q$ it is easy to see that 
	\begin{align}\label{ineq:2}
		\left(\|\nabla u\|_p^p+ \|\nabla u\|_{q,a}^q\right) \leq q\phi_\mathcal{H}(\nabla u).
	\end{align}

Using \eqref{ineq:2} in \eqref{ksdp-14}, H\"older's inequality and the best Sobolev constant $S$ defined in \eqref{best-sobolev-constant}, it follows that
	\begin{align*}
		& b_0\phi_\mathcal{H}^{\theta-1}(\nabla u)\left[(r-p-q(\theta-1))\|\nabla u\|_p^p+ (r-q-q(\theta-1)) \|\nabla u\|_{q,a}^q\right]\\
		&\leq a_0 \left[(r-p)\|\nabla u\|_p^p+(r-q) \|\nabla u\|_{q,a}^q\right]\\
		&\quad+ b_0\phi_\mathcal{H}^{\theta-1}(\nabla u)\left[(r-p-q(\theta-1))\|\nabla u\|_p^p+ (r-q-q(\theta-1)) \|\nabla u\|_{q,a}^q\right]\\
		& \leq \left[a_0 +b_0\phi_\mathcal{H}^{\theta-1}(\nabla u)\right] \left[(r-p)\|\nabla u\|_p^p+(r-q) \|\nabla u\|_{q,a}^q\right]\\
		&\quad- b_0(\theta -1)\phi_\mathcal{H}^{\theta-2}(\nabla u)\left(\|\nabla u\|_p^p+ \|\nabla u\|_{q,a}^q\right)^2  \\
		&=  (r-1+\gamma) \lambda \int_\Omega |u|^{1-\gamma}\,\diff x\\
		& \leq  (r-1+\gamma) \lambda |\Omega|^{1-\frac{1-\gamma}{p^*}}S^{-\frac{1-\gamma}{p}}\|\nabla u\|_p^{1-\gamma}.
	\end{align*}
	Again using $p \phi_{\mathcal{H}}(\nabla u) \geq \|\nabla u\|_p^p$, this implies
	\begin{align*}
		\frac{b_0(r-p-q(\theta-1))}{p^{\theta-1}}\|\nabla u\|_p^{p\theta}\leq (r-1+\gamma) \lambda |\Omega|^{1-\frac{1-\gamma}{p^*}}S^{-\frac{1-\gamma}{p}}\|\nabla u\|_p^{1-\gamma}.
	\end{align*}
	From $q\theta<r$ and $p<q$ we conclude that
	\begin{align}\label{ksdp-15}
		\|\nabla u\|_p \leq \left( \frac{(r-1+\gamma) \lambda |\Omega|^{1-\frac{1-\gamma}{p^*}}S^{-\frac{1-\gamma}{p}}p^{\theta-1} }{b_0(r-p-q(\theta-1))}\right)^{\frac{1}{p\theta-1+\gamma}}=:C\lambda^{\frac{1}{p\theta-1+\gamma}}\quad\text{with }C>0.
	\end{align}

	Using \eqref{ksdp-15} along with $\theta < \frac{r}{q}< \frac{r}{p}$, we get from \eqref{estimate:1} that
	\begin{align*}
		T_\lambda (u) \geq 
		\|\nabla u\|^r_p \left( A
			(C \lambda^{\frac{1}{p\theta-1+\gamma}} )^{p\theta-r} -B\right).
	\end{align*} 
	Setting
	\begin{align*}
		\Lambda_1 := \left(\frac{A}{BC^{r-p\theta}}\right)^{\frac{p\theta-1+\gamma}{r-p\theta}},
	\end{align*}
	we see that $T_\lambda (u) >0$ whenever $\lambda \in (0,\Lambda_1)$ contradicting the fact that $T_\lambda (u)=0$ for all $u\in \mathcal{N}_\lambda^\circ$. This proves the result.
\end{proof}

Let us now analyze the map $\psi_u'(t)$ in more detail. First, we can write 
\begin{align}\label{ksdp-6}
	\psi_u'(t)= t^{-\gamma} \left(\sigma_u(t) -\lambda \int_\Omega |u|^{1-\gamma}\,\diff x\right), \quad t>0,
\end{align}
where 
\begin{align*}
	\sigma_u(t)= \left[a_0 +b_0\phi_\mathcal{H}^{\theta-1}(t\nabla u)\right] \left(t^{p-1+\gamma}\|\nabla u\|_p^p+t^{q-1+\gamma} \|\nabla u\|_{q,a}^q\right)
	- t^{r-1+\gamma}\int_\Omega |u|^r\,\diff x.
\end{align*}

From the definition in \eqref{ksdp-6} it is clear that $tu\in \mathcal{N}_\lambda $ if and only if
\begin{align}\label{sigma-4}
	\sigma_u(t)= \lambda\int_\Omega |u|^{1-\gamma}\,\diff x.
\end{align}
The next lemma shows that the sets $\mathcal{N}_\lambda^+$ and $\mathcal{N}_\lambda^-$ are nonempty, whenever $\lambda$ is sufficiently small.

\begin{lemma}\label{ksdp-lem2}
	Let hypotheses \textnormal{(H)} be satisfied and let $u\in \V\setminus\{0\}$. Then there exist $\Lambda_2>0$ and unique $t_1^u<t_{\max}^u<t_2^u$ such that 
	\begin{align*}
		t_1^u u \in \mathcal{N}_\lambda^+,\quad t_2^u u\in \mathcal{N}_\lambda^-
		\quad\text{and}\quad \sigma_u(t_{\max}^u)=\max_{t>0} \sigma_u(t)
	\end{align*}
	whenever $\lambda \in (0,\Lambda_2)$.
\end{lemma}

\begin{proof}
	Let $u\in \V\setminus\{0\}$. The equation
	\begin{align*}
		0=\sigma_u'(t)
		&= \left[a_0 +b_0\phi_\mathcal{H}^{\theta-1}(t\nabla u)\right]\left[(p-1+\gamma)t^{p-2+\gamma}\|\nabla u\|_p^p+ (q-1+\gamma)t^{q-2+\gamma} \|\nabla u\|_{q,a}^q\right] \\
		&\qquad + b_0(\theta-1) \phi_\mathcal{H}^{\theta-2}(t\nabla u)\left(t^{p-1+\gamma}\|\nabla u\|_p^p+t^{q-1+\gamma} \|\nabla u\|_{q,a}^q\right)\left(t^{p-1}\|\nabla u\|_p^p+t^{q-1} \|\nabla u\|_{q,a}^q\right) \\
		&\qquad - (r-1+\gamma)t^{r-2+\gamma}\int_\Omega |u|^r\,\diff x
	\end{align*}
	is equivalent to
	\begin{align}\label{sigma-1}
		\begin{split}
			&\left[a_0 +b_0\phi_\mathcal{H}^{\theta-1}(t\nabla u)\right]\left[(p-1+\gamma)t^{p-r}\|\nabla u\|_p^p+ (q-1+\gamma)t^{q-r} \|\nabla u\|_{q,a}^q\right] \\
			&\quad + b_0(\theta-1) \phi_\mathcal{H}^{\theta-2}(t\nabla u)\left(t^{p-r+1}\|\nabla u\|_p^p+t^{q-r+1} \|\nabla u\|_{q,a}^q\right)\left(t^{p-1}\|\nabla u\|_p^p+t^{q-1} \|\nabla u\|_{q,a}^q\right) \\
			&= (r-1+\gamma)\int_\Omega |u|^r\,\diff x.
		\end{split}
	\end{align}
	Note that $r>q\theta$ and $\theta\geq 1$ imply
	\begin{align}\label{sigma-2}
		\begin{split}
			p(\theta-1)+p-r
			& < \min\left\{p(\theta-1)+q-r,q(\theta-1)+p-r\right\}\\
			&\leq \max\left\{p(\theta-1)+q-r,q(\theta-1)+p-r\right\}\\
			&<q(\theta-1)+q-r=q\theta-r<0.
		\end{split}	
	\end{align}
	Denoting the left-hand side of \eqref{sigma-1} as
	$$
	\begin{aligned}
			T_u(t)=&\left[a_0 +b_0\phi_\mathcal{H}^{\theta-1}(t\nabla u)\right]\left[(p-1+\gamma)t^{p-r}\|\nabla u\|_p^p+ (q-1+\gamma)t^{q-r} \|\nabla u\|_{q,a}^q\right] \\
			&+ b_0(\theta-1) \phi_\mathcal{H}^{\theta-2}(t\nabla u)\left(t^{p-r+1}\|\nabla u\|_p^p+t^{q-r+1} \|\nabla u\|_{q,a}^q\right)\left(t^{p-1}\|\nabla u\|_p^p+t^{q-1} \|\nabla u\|_{q,a}^q\right)
\end{aligned}
	$$
and using \eqref{sigma-2} as well as $0<\gamma<1<p<q<r$, we easily observe that
	\begin{enumerate}
		\item[\textnormal{(i)}]
			$\displaystyle\lim_{t\to0^+} T_u(t)=\infty$;
		\item[\textnormal{(ii)}]
			$\displaystyle\lim_{t\to\infty} T_u(t)=0$;
		\item[\textnormal{(iii)}]
			$T_u'(t)<0$ for all $t>0$.
	\end{enumerate}
	From \textnormal{(i)} and \textnormal{(ii)} along with the intermediate value theorem there exists $t_{\max}^u>0$ such that \eqref{sigma-1} holds. From \textnormal{(iii)} we see that $t_{\max}^u$ is unique due to the injectivity of $T_u$. Moreover, if we consider $\sigma'_u(t)>0$, then in place of \eqref{sigma-1} we get
$$T_u(t)> (r-1+\gamma)\int_\Omega |u|^r\,\diff x
$$
and since $T_u$ is strictly decreasing, this holds for all $t<t_{\max}^u$. The same can be said for $\sigma_u'(t)<0$ and $t>t_{\max}^u$. Therefore, $\sigma_u$ is injective in $(0,t_{\max}^u)$ and in $(t_{\max}^u,\infty)$. In addition,
	\begin{align*}
		\sigma_u(t_{\max}^u)=\max_{t>0} \sigma_u(t)
	\end{align*}
	with $t_{\max}^u>0$ being the global maximum of $\sigma_u$. Moreover, we have
	\begin{align*}
		\lim_{t\to 0^+} \sigma_u(t) = 0
		\quad\text{and}\quad 
		\lim_{t\to \infty} \sigma_u(t) = -\infty.
	\end{align*}
	
	Using again $p \phi_{\mathcal{H}}(\nabla u) \geq \|\nabla u\|_p^p$ we observe that
	\begin{align*}
		\sigma_u'(t)
		\geq \frac{b_0}{p^{\theta-1}}(p-1+\gamma)t^{p\theta-2+\gamma}\|\nabla u\|_p^{p\theta} - (r-1+\gamma)t^{r-2+\gamma}\int_\Omega |u|^r\,\diff x,
	\end{align*}
	which gives by applying H\"older's inequality and \eqref{best-sobolev-constant} that
	\begin{align}\label{sigma-3}
		t_{\max}^u \geq  \frac{1}{\|\nabla u\|_p}\left( \frac{b_0(p-1+\gamma)S^{\frac{r}{p}}}{p^{\theta-1}(r-1+\gamma)|\Omega|^{1-\frac{r}{p^*}}}\right)^{\frac{1}{r-p\theta}} := t_0^{u}.
	\end{align} 
	Since $\sigma_u$ is increasing on $(0,t_{\max}^u)$, we obtain from $p \phi_{\mathcal{H}}(\nabla u) \geq \|\nabla u\|_p^p$,  H\"older's inequality, \eqref{best-sobolev-constant} and the representation of  $t_0^u$ in \eqref{sigma-3} that
	\begin{align*}
		 \sigma_u(t_{\max}^u) &\geq \sigma_u(t_0^u) 
		 \geq \frac{b_0}{p^{\theta-1}} (t_0^u)^{p \theta -1+\gamma}\|\nabla u\|_p^{p\theta} -  (t_0^u)^{r-1+\gamma}\int_\Omega |u|^r\,\diff x\\
		 &\geq (t_0^u)^{p \theta -1+\gamma} \|\nabla u\|_p^{p\theta} \left(\frac{b_0}{p^{\theta-1}}- (t_0^u)^{r-p\theta} S^{-\frac{r}{p}}|\Omega|^{1-\frac{r}{p^*}} \|\nabla u\|^{r-p\theta}_p\right)\\
		 & \geq \left(\frac{r-p}{r-1+\gamma}\right)\frac{b_0}{p^{\theta-1}} (t_0^u)^{p \theta -1+\gamma} \|\nabla u\|_p^{p\theta}\\
		 &\geq \left(\frac{r-p}{r-1+\gamma}\right) \|\nabla u\|^{1-\gamma}_p\frac{b_0}{p^{\theta-1}}\left(\frac{b_0(p-1+\gamma)S^{\frac{r}{p}}}{p^{\theta-1}(r-1+\gamma)|\Omega|^{1-\frac{r}{p^*}}}\right)^{\frac{p\theta-1+\gamma}{r-p\theta}}\\
		 & \geq \Lambda_2 \int_\Omega |u|^{1-\gamma}\,\diff x,
	\end{align*}
	where
	\begin{align*}
		\Lambda_2=\frac{b_0}{p^{\theta-1}} \left(\frac{r-p}{r-1+\gamma}\right)\left(\frac{b_0(p-1+\gamma)S^{\frac{r}{p}}}{p^{\theta-1}(r-1+\gamma)|\Omega|^{1-\frac{r}{p^*}}}\right)^{\frac{p\theta-1+\gamma}{r-p\theta}} \frac{S^{\frac{1-\gamma}{p}}}{|\Omega|^{\frac{p^*+\gamma-1}{p^*}}}.
	\end{align*}
	From the considerations above, we see that
	\begin{align*}
		\sigma_u(t_{\max}^u) > \lambda \int_\Omega |u|^{1-\gamma}\,\diff x
	\end{align*}
	whenever $\lambda \in (0,\Lambda_2)$.
	
	Recall that $\sigma_u$ is injective in $(0,t_{\max}^u)$ and in $(t_{\max}^u,\infty)$. Hence, we find unique $t_1^u, t_2^u>0$ such that
	\begin{align*}
		\sigma_u(t_1^u) = \lambda \int_\Omega |u|^{1-\gamma}\,\diff x = \sigma_u(t_2^u)
		\quad\text{with}\quad \sigma_u'(t_2^u)<0<\sigma_u'(t_1^u).
	\end{align*}
	In addition, we have $t_1^u u$, $t_2^u u \in \mathcal{N}_\lambda$, see \eqref{sigma-4}. Using the representation in \eqref{ksdp-6} we observe that
	\begin{align*}
		\sigma_u'(t) = t^\gamma \psi_u''(t)+\gamma t^{\gamma-1}\psi_u'(t).
	\end{align*}
	Since $\psi_u'(t_1^u)=\psi_u'(t_2^u)=0$ and $\sigma_u'(t_2^u)<0<\sigma_u'(t_1^u)$, we get
	\begin{align*}
		0<\sigma_u'(t_1^u)= (t_1^u)^\gamma \psi_u''(t_1^u)
		\quad\text{and}\quad 
		0>\sigma_u'(t_2^u)= (t_2^u)^\gamma \psi_u''(t_2^u).
	\end{align*}
	Hence, $t_1^u u\in \mathcal{N}_\lambda^+$ and $t_2^u u\in \mathcal{N}_\lambda^-$ which completes the proof.
\end{proof}

Next we prove lower and upper bounds for the modular $\varrho_\mathcal{H}(\nabla \cdot)$ for the elements of $\mathcal{N}_\lambda^+$ and $\mathcal{N}_\lambda^-$, respectively.

\begin{proposition}\label{gap-struc}
	Let hypotheses \textnormal{(H)} be satisfied and let $\lambda >0$.  Then there exist constants $D_1=D_1(\lambda)>0$ and $D_2>0$ such that
	\begin{align*}
		\|\nabla u\|_p^p + \|\nabla u\|_{q,a}^q<D_1
		\quad\text{and}\quad  
		\|\nabla v\|_p^p >D_2
	\end{align*} 
	for every $u\in \mathcal{N}_\lambda^+$ and for every $v\in \mathcal{N}_\lambda^-$.
\end{proposition}

\begin{proof}
	Let $u\in \mathcal{N}_\lambda^+$. From $\psi_u'(1)=0$ and $\psi_u''(1)>0$ we get 
	\begin{align*}
		&(r-1)[a_0+b_0 \phi_\mathcal{H}^{\theta-1}(\nabla u)](\|\nabla u\|_p^p + \|\nabla u\|_{q,a}^q)-\lambda(r-1)\int_\Omega |u|^{1-\gamma}\,\diff x\\
		&< [a_0+b_0\phi_\mathcal{H}^{\theta-1}(\nabla u)]((p-1)\|\nabla u\|_p^p + (q-1)\|\nabla u\|_{q,a}^q)\\
		& \quad + b_0(\theta-1)\phi_\mathcal{H}^{\theta-2}(\nabla u)(\|\nabla u\|_p^p + \|\nabla u\|_{q,a}^q)^2 + \lambda \gamma \int_\Omega |u|^{1-\gamma}\,\diff x.
	\end{align*}
	Using $\phi_\mathcal{H}(\nabla u)\geq \frac{1}{q}(\|\nabla u\|_p^p + \|\nabla u\|_{q,a}^q)$ in the inequality above along with H\"older's inequality and \eqref{best-sobolev-constant} it follows 
	\begin{align}\label{gap-struc-1}
		\begin{split}
			&a_0 ((r-p)\|\nabla u\|_p^p + (r-q)\|\nabla u\|_{q,a}^q)\\
			&\quad+b_0\phi_\mathcal{H}^{\theta-1}(\nabla u)((r-p-q(\theta-1))\|\nabla u\|_p^p + (r-q\theta)\|\nabla u\|_{q,a}^q)\\
			& < \lambda (r-1+\gamma) \int_\Omega |u|^{1-\gamma}\,\diff x \leq \lambda (r-1+\gamma) |\Omega|^{1-\frac{1-\gamma}{p^*}} S^{-\frac{1-\gamma}{p}}\|\nabla u\|_p^{1-\gamma}.
		\end{split}
	\end{align}
	Since $r>\theta q$ we have $r-p-q(\theta-1) \geq r-q-q(\theta-1)>0$. Hence, we obtain from \eqref{gap-struc-1}
	\begin{align*}
		\lambda (r-1+\gamma)|\Omega|^{1-\frac{1-\gamma}{p^*}} S^{-\frac{1-\gamma}{p}} 
		> \frac{b_0}{p^{\theta-1}}(r-p-q(\theta-1))\|\nabla u\|_p^{p\theta-1+\gamma},
	\end{align*}
	which gives
	\begin{align}\label{gap-struc-n}
		\|\nabla u\|_p^p < A_1:= 
		\left(\frac{\lambda p^{\theta-1} (r-1+\gamma)|\Omega|^{1-\frac{1-\gamma}{p^*}} S^{-\frac{1-\gamma}{p}}}{b_0(r-p-q(\theta-1))}\right)^{\frac{p}{p\theta-1+\gamma}}.
	\end{align}
	Putting \eqref{gap-struc-n} in \eqref{gap-struc-1}, we get 
	\begin{align*}
		\lambda (r-1+\gamma)|\Omega|^{1-\frac{1-\gamma}{p^*}} S^{-\frac{1-\gamma}{p}}A_1^{\frac{1-\gamma}{p}} 
		>\frac{b_0}{q^{\theta-1}}(r-q\theta)\|\nabla u\|_{q,a}^{q\theta},
	\end{align*}
	which results in
	\begin{align}\label{gap-struc-2}
		\|\nabla u\|_{q,a}^q  < A_2:= 
		\left(\frac{\lambda q^{\theta-1} (r-1+\gamma)|\Omega|^{1-\frac{1-\gamma}{p^*}} S^{-\frac{1-\gamma}{p}}A_1^{\frac{1-\gamma}{p}}}{b_0(r-q\theta)}\right)^\frac{1}{\theta}.
	\end{align}
	From \eqref{gap-struc-n} and \eqref{gap-struc-2} we conclude that
	\begin{align*}
		\|\nabla u\|_p^p +\|\nabla u\|_{q,a}^q < A_1+A_2=:D_1.
	\end{align*}

	Next let us we fix $v\in \mathcal{N}_\lambda^-$. Then $\psi_v'(1)=0$ and $\psi_v''(1)<0$ gives us
	\begin{align*}
		&[a_0+b_0\phi_\mathcal{H}^{\theta-1}(\nabla v)]((p-1)\|\nabla v\|_p^p + (q-1)\|\nabla v\|_{q,a}^q)+\gamma [a_0+b_0\phi_\mathcal{H}^{\theta-1}(\nabla v)](\|\nabla v\|_p^p + \|\nabla v\|_{q,a}^q)\\
		&\quad + b_0(\theta-1)\phi_\mathcal{H}^{\theta-2}(\nabla v)(\|\nabla v\|_p^p + \|\nabla v\|_{q,a}^q)^2\\
		&<  (r-1+\gamma) \int_\Omega |v|^r\,\diff x\leq  (r-1+\gamma) |\Omega|^{1-\frac{r}{p^*}} S^{-\frac{r}{p}}\|\nabla v\|^r_p.
	\end{align*}
	Therefore, we get 
	\begin{align*}
		\|\nabla v\|_p^p > D_2:=
		\left(\frac{b_0(p-1)}{p^{\theta-1}(r-1+\gamma) |\Omega|^{1-\frac{r}{p^*}} S^{-\frac{r}{p}}}\right)^{\frac{p}{r-p\theta}}.
	\end{align*}
	Thus, the proof is finished.
\end{proof}

\section{Existence of solutions for problem $(P_\lambda)$}\label{sec_4}

In this section we use the results from Section \ref{sec_3} in order to prove Theorem \ref{main_result}. To this end, we first define
\begin{align*}
	\Theta_\lambda ^+ = \inf_{u\in\mathcal{N}_\lambda^+}J_\lambda(u).
\end{align*}

The next proposition shows that this minimum is achieved and it has negative energy.

\begin{proposition}\label{mini:N+}
	Let hypotheses \textnormal{(H)} be satisfied and let $\lambda \in (0,\min\{\Lambda_1,\Lambda_2\})$, with $\Lambda_1$, $\Lambda_2$ given in Lemmas \ref{ksdp-lem3-} and \ref{ksdp-lem2}. Then $\Theta_\lambda ^+<0$ and there exists $u_\lambda \in \mathcal{N}_\lambda^+$ such that $J_\lambda(u_\lambda ) = \Theta_\lambda ^+<0$ with $u_\lambda\geq 0$ a.\,e.\,in $\Omega$.
\end{proposition}

\begin{proof}
	Let $u\in \mathcal{N}_\lambda^+$. Then we have 
	\begin{align}\label{mini:N+:1}
		-\left[a_0 +b_0\phi_\mathcal{H}^{\theta-1}(\nabla u)\right] \left(\|\nabla u\|_p^p+  \|\nabla u\|_{q,a}^q\right)+\int_\Omega |u|^r\,\diff x=-\lambda \int_\Omega |u|^{1-\gamma}\,\diff x
	\end{align}
	and
	\begin{align} \label{mini:N+:2}
		\begin{split}
			& \left[a_0 +b_0\phi_\mathcal{H}^{\theta-1}(\nabla u)\right] \left[(p-1)\|\nabla u\|_p^p+(q-1) \|\nabla u\|_{q,a}^q\right]\\
			&\quad+ b_0(\theta -1)\phi_\mathcal{H}^{\theta-2}(\nabla u)\left(\|\nabla u\|_p^p+ \|\nabla u\|_{q,a}^q\right)^2  + \lambda \gamma \int_\Omega |u|^{1-\gamma}\,\diff x >(r-1) \int_\Omega |u|^r\,\diff x.
		\end{split}
	\end{align}
	Combining \eqref{mini:N+:1} multiplied with $-\gamma$ and \eqref{mini:N+:2} yields
	\begin{align}\label{mini:N+:6}
		\begin{split}
			(r-1+\gamma) \int_\Omega |u|^r\,\diff x 
			&< \left[a_0 +b_0\phi_\mathcal{H}^{\theta-1}(\nabla u)\right] \left[(p-1+\gamma)\|\nabla u\|_p^p+(q-1+\gamma) \|\nabla u\|_{q,a}^q\right]\\
			&\quad+ b_0(\theta -1)\phi_\mathcal{H}^{\theta-2}(\nabla u)\left(\|\nabla u\|_p^p+ \|\nabla u\|_{q,a}^q\right)^2.
		\end{split}
	\end{align}
	Using \eqref{mini:N+:1} and \eqref{mini:N+:6} we obtain
	\begin{align*}
		J_\lambda(u) &= \left[a_0\phi_\mathcal{H}(\nabla u) +\frac{b_0}{\theta}\phi_\mathcal{H}^{\theta}(\nabla u)\right]- \frac{\lambda}{1-\gamma}\int_\Omega |u|^{1-\gamma}\,\diff x - \frac{1}{r}\int_\Omega |u|^r\,\diff x\\
		& = \left[a_0\phi_\mathcal{H}(\nabla u) +\frac{b_0}{\theta}\phi_\mathcal{H}^{\theta}(\nabla u)\right]-\frac{1}{1-\gamma}\left[a_0 +b_0\phi_\mathcal{H}^{\theta-1}(\nabla u)\right] \left(\|\nabla u\|_p^p+  \|\nabla u\|_{q,a}^q\right)\\
		&\quad + \frac{r-1+\gamma}{r(1-\gamma)}\int_\Omega |u|^r\,\diff x\\
		&< \left[a_0\phi_\mathcal{H}(\nabla u) +\frac{b_0}{\theta}\phi_\mathcal{H}^{\theta}(\nabla u)\right]+\left[a_0 +b_0\phi_\mathcal{H}^{\theta-1}(\nabla u)\right]\left[\left(\frac{p-1+\gamma}{r(1-\gamma)}-\frac{1}{1-\gamma}\right)\|\nabla u\|_p^p\right.\\
		&\quad \left. +\left(\frac{q-1+\gamma}{r(1-\gamma)}-\frac{1}{1-\gamma}\right)\|\nabla u\|_{q,a}^q\right] +b_0\frac{(\theta -1)}{r(1-\gamma)}\phi_\mathcal{H}^{\theta-2}(\nabla u)\left(\|\nabla u\|_p^p+ \|\nabla u\|_{q,a}^q\right)^2\\
		& = a_0 \left[\left(\frac{1}{p}+\frac{p-1+\gamma-r}{r(1-\gamma)}\right)\|\nabla u\|_p^p+ \left(\frac{1}{q}+\frac{q-1+\gamma-r}{r(1-\gamma)}\right)\|\nabla u\|_{q,a}^q\right]\\
		&\quad +b_0 \left[\frac{\phi_\mathcal{H}^\theta(\nabla u)}{\theta}+\frac{\phi_\mathcal{H}^{\theta-1}(\nabla u)(p-1+\gamma-r)}{r(1-\gamma)}\|\nabla u\|_p^p+\frac{\phi_\mathcal{H}^{\theta-1}(\nabla u)(q-1+\gamma-r)}{r(1-\gamma)}\|\nabla u\|_{q,a}^q\right.\\
		&\qquad\quad\quad \left.+\frac{(\theta -1)}{r(1-\gamma)}\phi_\mathcal{H}^{\theta-2}(\nabla u)\left(\|\nabla u\|_p^p+ \|\nabla u\|_{q,a}^q\right)^2\right]\\
		& =a_0B_1+b_0B_2
	\end{align*}
	with
	\begin{align*}
		B_1&=\left(\frac{1}{p}+\frac{p-1+\gamma-r}{r(1-\gamma)}\right)\|\nabla u\|_p^p+ \left(\frac{1}{q}+\frac{q-1+\gamma-r}{r(1-\gamma)}\right)\|\nabla u\|_{q,a}^q,\\
		B_2&=\frac{\phi_\mathcal{H}^\theta(\nabla u)}{\theta}+\frac{\phi_\mathcal{H}^{\theta-1}(\nabla u)(p-1+\gamma-r)}{r(1-\gamma)}\|\nabla u\|_p^p+\frac{\phi_\mathcal{H}^{\theta-1}(\nabla u)(q-1+\gamma-r)}{r(1-\gamma)}\|\nabla u\|_{q,a}^q\\
		&\quad +\frac{(\theta -1)}{r(1-\gamma)}\phi_\mathcal{H}^{\theta-2}(\nabla u)\left(\|\nabla u\|_p^p+ \|\nabla u\|_{q,a}^q\right)^2.
	\end{align*}
	From the assumptions \textnormal{(H)} we have
	\begin{align*}
		\frac{1}{p}+\frac{p-1+\gamma-r}{r(1-\gamma)}&= \frac{(p-r)(p-1+\gamma)}{pr(1-\gamma)}<0,\\
		\frac{1}{q}+\frac{q-1+\gamma-r}{r(1-\gamma)}&= \frac{(q-r)(q-1+\gamma)}{qr(1-\gamma)}\leq 0.
	\end{align*}
	Therefore, $B_1<0$.
	
	Let us consider $B_2$. Using
	\begin{align*}
		\phi_\mathcal{H}^{\theta-2}(\nabla u)\left(\|\nabla u\|_p^p+ \|\nabla u\|_{q,a}^q\right)^2\leq q\phi_\mathcal{H}^{\theta-1}(\nabla u)\left(\|\nabla u\|_p^p+ \|\nabla u\|_{q,a}^q\right),
	\end{align*}
	we get
	\begin{align*}
		B_2 
		&\leq \frac{\phi_\mathcal{H}^\theta(\nabla u)}{\theta}+\phi_\mathcal{H}^{\theta-1}(\nabla u)\frac{(p-1+\gamma-r)+q(\theta-1)}{r(1-\gamma)}\|\nabla u\|_p^p\\
		&\quad +\phi_\mathcal{H}^{\theta-1}(\nabla u)\frac{(q-1+\gamma-r)+q(\theta-1)}{r(1-\gamma)}\|\nabla u\|_{q,a}^q\\
		& =\phi_\mathcal{H}^{\theta-1}(\nabla u)\left [ \left(\frac{1}{p\theta}+ \frac{(p-1+\gamma-r)+q(\theta-1)}{r(1-\gamma)}\right)\|\nabla u\|_p^p  \right.\\
		&\left.\quad +\left(\frac{1}{q\theta}+ \frac{(q\theta-1+\gamma-r)}{r(1-\gamma)}\right)\|\nabla u\|_{q,a}^q\right].
	\end{align*}
	Now, since $\theta\geq 1$ and $\theta q<r$, we obtain 
	\begin{align*}   
		\frac{1}{p\theta}+ \frac{(p-1+\gamma-r)+q(\theta-1)}{r(1-\gamma)}  
		&= \frac{r(1-\gamma)+p\theta(p-1+\gamma -r)+p\theta q(\theta-1)-rp+rp}{p\theta r(1-\gamma)}\\
		& =\frac{(p\theta-r)(p-1+\gamma)+p(\theta-1)(\theta q-r)}{p\theta r(1-\gamma)}<0,
	\end{align*}
	because $\theta p<\theta q<r$. Similarly we have
	\begin{align*}
		\frac{1}{q\theta}+ \frac{(q\theta-1+\gamma-r)}{r(1-\gamma)} 
		&= \frac{r(1-\gamma)+q\theta(q\theta-1+\gamma-r)-q\theta r+q \theta r}{q\theta r(1-\gamma)}\\
		& =\frac{(q\theta-r)(q\theta-1+\gamma)}{q\theta r(1-\gamma)}<0.
	\end{align*}
	From the considerations above it follows that $B_2<0$. Hence $J_\lambda(u)<0$ which implies that
	\begin{align*}
		\Theta_\lambda ^+ \leq J_\lambda(u)<0.
	\end{align*}

	Let us now prove the second part of the proposition. To this end, let $\{u_n\}_{n\in\N}\subset \mathcal{N}_\lambda^+$ be a minimizing sequence, that is, 
	\begin{align*}
		J_\lambda(u_n)\searrow \Theta_\lambda^+<0\quad\text{as }n\to\infty.
	\end{align*}
	By Lemma \ref{ksdp-lem1}, we know that $\{u_n\}_{n\in\N}$ is bounded in $\V$. Hence, by Proposition \ref{proposition_embeddings}\textnormal{(ii)} along with the reflexivity of $\V$, there exist a subsequence still denoted by $\{u_n\}_{n\in\N}$ and $u_\lambda\in\V$ such that
	\begin{align}\label{mini:N+:4}
		u_n\rightharpoonup u_\lambda \quad\text{in }\V,\quad
		u_n\to u_\lambda\quad\text{in }\Lp{s}
		\quad\text{and}\quad u_n\to u_\lambda \quad \text{a.\,e.\,in }\Omega 
	\end{align}
	for any $s\in [1,p^*)$. Applying the weak lower semicontinuity of the norms and seminorms and Lebesgue's dominated convergence theorem, we infer that
	\begin{align*}
		J_\lambda(u_\lambda ) \leq \liminf_{n\to \infty} J_\lambda(u_n) = \Theta_\lambda ^+<0=J_\lambda(0).
	\end{align*}
	Thus, $u_\lambda \neq 0$. From Lemma \ref{ksdp-lem2}, we know that there exists a unique $t_1^{u_\lambda}>0$ such that $t_1^{u_\lambda} u_\lambda \in \mathcal{N}_\lambda^+$. \\
	
	{\bf Claim:} $u_n\to u_\lambda $ \textit{in} $\V$

	Let us suppose by contradiction that 
	\begin{align*}
		\liminf_{n\to  \infty} \|\nabla u_n\|_p^p> \|\nabla u_\lambda \|_p^p.
	\end{align*}
	From this inequality along with $\psi_{u_\lambda }'(t_1^{u_\lambda })=0$, \eqref{mini:N+:4}, the weak lower semicontinuity of the norms and seminorms and Lebesgue's dominated convergence theorem, we infer that
	\begin{align*}
		\liminf_{n\to \infty} \psi_{u_n}'(t_1^{u_\lambda}) 
		&= \liminf_{n\to \infty}\bigg[ \left[(t_1^{u_\lambda})^{p-1} \|\nabla u_n\|_p^p+ (t_1^{u_\lambda })^{q-1}\|\nabla u_n\|_{q,a}^q\right]\left(a_0 + b_0\phi_\mathcal{H}^{\theta-1}(t_1^{u_\lambda }\nabla u_n)\right)\\
		&\qquad\qquad \quad  -\lambda (t_1^{u_\lambda })^{-\gamma}\int_\Omega |u_n|^{1-\gamma}\,\diff x - (t_1^{u_\lambda })^{r-1}\int_\Omega |u_n|^r\,\diff x\bigg]\\
		&> \left[(t_1^{u_\lambda})^{p-1}\|\nabla u_\lambda \|_p^p+ (t_1^{u_\lambda})^{q-1}\|\nabla u_\lambda \|_{q,a}^q\right]\left(a_0 + b_0\phi_\mathcal{H}^{\theta-1}(t_1^{u_\lambda}\nabla u_\lambda )\right)\\
		&\quad  -\lambda (t_1^{u_\lambda})^{-\gamma}\int_\Omega |u_\lambda|^{1-\gamma}\,\diff x - (t_1^{u_\lambda})^{r-1}\int_\Omega |u_\lambda|^r\,\diff x\\
		& = \psi_{u_\lambda }'(t_1^{u_\lambda})=0.
	\end{align*}
	Thus, we can find a number $n_0\in \N$ such that
	\begin{align*}
		\psi_{u_n}'(t_1^{u_\lambda})>0\quad \text{for all } n \geq n_0.
	\end{align*}
	Recalling the representation in \eqref{ksdp-6}, we get that $\psi_{u_{n}}'(t)<0$ for $t\in(0,1)$, and since $\psi_{u_n}'(1)=0$, this implies $t_1^{u_\lambda}>1$. Hence, $t_1^{u_\lambda} u_\lambda \in \mathcal{N}_\lambda^+$ gives that
	\begin{align*}
		\Theta_\lambda ^+\leq { J_\lambda(t_1^{u_\lambda} u_\lambda )\leq J_\lambda(u_\lambda )}< \liminf_{n\to\infty}J_\lambda(u_n)=\Theta_\lambda ^+
	\end{align*}
	which is a contradiction. Therefore, we find a subsequence such that
	\begin{align*}
		\lim_{n\to  \infty} \|\nabla u_n\|_p^p= \|\nabla u_\lambda \|_p^p.
	\end{align*}

	If we suppose that
	\begin{align*}
		\lim_{n\to  \infty} \|\nabla u_n\|_{q,a}^q> \|\nabla u_\lambda \|_{q,a}^q,
	\end{align*}
	then we can argue as above reaching a contradiction. Therefore, for a subsequence, we have
	\begin{align*}
		\lim_{n\to  \infty} \|\nabla u_n\|_{q,a}^q= \|\nabla u_\lambda \|_{q,a}^q.
	\end{align*}

	From these considerations we obtain that $\varrho_{\mathcal{H}}(\nabla u_n)\to \varrho_{\mathcal{H}}(\nabla u_\lambda)$ and since the integrand corresponding to the modular function is uniformly convex, we get that $\varrho_{\mathcal{H}}(\frac{\nabla u_n-\nabla u_\lambda}{2})\to 0$. From Proposition \ref{proposition_modular_properties}\textnormal{(v)} we obtain that $u_n\to u_\lambda$ in $\V$. The continuity of $J_\lambda$ implies that $J_\lambda(u_n)\to J_\lambda(u_\lambda)$ and so $J_\lambda(u_\lambda)=\Theta_\lambda^+$.

	Finally, we have to show that $u_\lambda \in \mathcal{N}_\lambda^+$. Since $u_n \in \mathcal{N}_\lambda^+$ for all $n\in \N$, it holds
	\begin{align}\label{eq:1}
		\begin{split}
			\psi_{u_n}'(1) 
			&= \left[a_0 +b_0\phi_\mathcal{H}^{\theta-1}(\nabla u_n)\right] \left(\|\nabla u_n\|_p^p+\|\nabla u_n\|_{q,a}^q\right)\\ 
			&\quad - \lambda \int_\Omega |u_n|^{1-\gamma}\,\diff x - \int_\Omega |u_n|^r\,\diff x=0
		\end{split}
	\end{align}
	and
	\begin{align}\label{ineq:1}
		\begin{split}
			\psi_{u_n}''(1)
			& =   \left[a_0 +b_0\phi_\mathcal{H}^{\theta-1}(\nabla u_n)\right] \left[(p-1)\|\nabla u_n\|_p^p+(q-1) \|\nabla u_n\|_{q,a}^q\right]\\ 
			& \quad + b_0(\theta -1)\phi_\mathcal{H}^{\theta-2}(\nabla u_n)\left(\|\nabla u_n\|_p^p+ \|\nabla u_n\|_{q,a}^q\right)^2\\
			&\quad + \lambda \gamma \int_\Omega |u_n|^{1-\gamma}\,\diff x -(r-1) \int_\Omega |u_n|^r\,\diff x>0. 
		\end{split}
	\end{align}
	Passing to the limit as $n\to \infty$ in \eqref{eq:1} and \eqref{ineq:1}, we get
	\begin{align*}
		\psi_{u_\lambda }'(1) =0
		\quad \text{and}\quad 
		\psi''_{u_\lambda }(1)\geq 0.
	\end{align*}
	Recall that $\lambda \in (0,\min\{\Lambda_1,\Lambda_2\})$. Then Lemma \ref{ksdp-lem3-} says that $\mathcal{N}_\lambda^\circ=\emptyset$ and so $\psi''_{u_\lambda }(1)> 0$. This shows that $u_\lambda \in \mathcal{N}_\lambda^+$. Noting that we can work with $|u_\lambda |$ instead of $u_\lambda $, we conclude that $u_\lambda \geq 0$ a.\,e.\,in $\Omega$.
\end{proof}

\begin{lemma}\label{ksdp-lem3}
	Let hypotheses \textnormal{(H)} be satisfied, $u\in \mathcal{N}_\lambda ^\pm$ and let $\lambda>0$. Then there exist $\varepsilon>0$ and a continuous function $\zeta\colon B_\varepsilon(0) \to (0,\infty)$ such that
	\begin{align*}
		\zeta(0)=1 \quad\text{and}\quad \zeta(v)(u+v)\in \mathcal{N}_\lambda ^\pm \quad \text{for all } v \in B_\varepsilon(0),
	\end{align*}
	where $B_\varepsilon(0) := \{v\in \V\,:\, \|v\|<\varepsilon\}$.
\end{lemma}

\begin{proof}
	We define the map $F\colon \V \times (0,\infty) \to \R$ given by
	\begin{align*}
		F(v,t)= t^{\gamma}\psi_{u+v }'(t)\quad \text{for }  (v,t)\in \V \times (0,\infty). 
	\end{align*}
	Note that
	\begin{align}\label{derivative-zeta}
		\frac{\partial F}{\partial t}(v,t)= \gamma t^{\gamma-1}\psi'_{u+v}(t) + t^\gamma \psi_{u+v}''(t).
	\end{align}
	Since $u\in \mathcal{N}_\lambda^+$, we obtain
	\begin{align}\label{ksdp-lem3-1}
		F(0,1) = \psi_u'(1)=0\quad\text{and}\quad \frac{\partial F}{\partial t}(0,1)=\psi_u''(1)>0.
	\end{align}
	Hence, we can apply the implicit function theorem to $F$ at $(0,1)$ (see, for example, Berger \cite[p.\,115]{Berger-1977}) to claim that there exists $\varepsilon>0$ such that for any $v\in \V$ with $\|v\|<\varepsilon$,  the equation $F(v,t)=0$ has a continuous unique solution $t=\zeta(v)>0$.  From this and \eqref{ksdp-lem3-1}, we infer that
	\begin{align*}
		\zeta(0)=1 \quad\text{and}\quad F(v,\zeta(v))=0 \quad \text{for all }v\in \V,
	\end{align*}
	whenever $\|v\|\leq \varepsilon$.  Therefore, $\zeta(v)(u+v)\in \mathcal{N}_\lambda $ for all $\|v\|\leq \varepsilon$ and from \eqref{derivative-zeta} we conclude that
	\begin{align*}
		\frac{\partial F}{\partial t}(v,\zeta (v))= (\zeta(v))^\gamma \psi''_{u+v}(\zeta (v))
		\quad \text{for all }\|v\| \leq \varepsilon.
	\end{align*}
	Recall that $\zeta(0)=1$ and $\frac{\partial F}{\partial t}(0,1)>0$. We observe that for $\delta \in (0,1)$ the mapping
	\begin{align*}
		\mathcal{F}_{u,\delta}\colon \V \times [1-\delta, 1+ \delta] \to \R \quad \text{defined as} \quad \mathcal{F}(v,\ell):=\psi''_{u+v}(\ell)
	\end{align*}
	is continuous. Hence, we can choose $\varepsilon>0$ small enough such that
	\begin{align*}
		\zeta(v)(u+v)\in \mathcal{N}_\lambda^+ \quad \text{for all }\|v\|\leq \varepsilon.
	\end{align*}
	The proof for the case $u\in \mathcal{N}_\lambda^-$ works similar.
\end{proof}

\begin{proposition}\label{ksdp-prop4}
	Let hypotheses \textnormal{(H)} be satisfied and let $\lambda \in (0,\min\{\Lambda_1,\Lambda_2\})$, with $\Lambda_1$, $\Lambda_2$ given in Lemmas \ref{ksdp-lem3-} and \ref{ksdp-lem2}. Then there exist $\varepsilon, \delta>0$ such that
	\begin{align*}
		J_\lambda(u_\lambda )\leq J_\lambda(u_\lambda +th) \quad \text{for all } h\in \V
	\end{align*} 
	when { $t\in P_\delta:=\{t \in [0,\delta]\,:\, th \in B_\varepsilon(0)\}$}, where $u_\lambda $ is as defined in Proposition \ref{mini:N+}.
\end{proposition}

\begin{proof}
	For $h\in \V$  we define the function $f_h\colon[0,\infty)\to \mathbb R$ given by
	\begin{align*}
		f_h(t)= \psi_{u_\lambda +th}''(1)\quad\text{for }t\in [0,\infty).
	\end{align*} 
	Since $u_\lambda \in \mathcal{N}_\lambda^+$,  we get
	\begin{align*}
		f_h(0)= \psi_{u_\lambda }''(1)>0.
	\end{align*}
	Recalling the expression of $\psi''_u$ in \eqref{second-derivative},  $u_\lambda \not \equiv0$ in $\Omega$ and by using the same arguments as in Lemma \ref{ksdp-lem3}, we assert that there exists $\delta_0>0$ such that
	\begin{align*}
		\psi_{u_\lambda +th}''(1)=f_h(t)>0\quad \text{for all } t\in [0,\delta_0].
	\end{align*}
	
	From Lemma \ref{ksdp-lem3} for $u_\lambda \in \mathcal{N}_\lambda^+$, we find $\varepsilon >0$ and a continuous map $\zeta\colon B_\varepsilon(0) \to (0,\infty)$ such that
	\begin{align*}
		\zeta(th)(u_\lambda +th)\in \mathcal{N}_\lambda^+ \quad \text{for all } t \in P_{\delta_0} \quad\text{with}\quad \zeta(t h)\to 1\quad\text{as } t\to 0^+.
	\end{align*} 
	This implies, in particular, the convexity of the function $\psi_{u_\lambda +th}''$ in a neighborhood of $1$ by taking $\delta_0$ small enough. The continuity of the map  $u_\lambda \mapsto \psi_{u_\lambda}''(1)$ and the fact that $\psi_{u_\lambda +th}'(\zeta(th)) =0$ allows us to choose $\delta \in (0, \delta_0)$ such that $\psi_{u_\lambda + th}''(1) >0$ and $\psi_{u_\lambda +th}(\zeta(t)) \leq \psi_{u_\lambda +th}(1)$ for $t \in P_{\delta}$. Therefore taking Proposition \ref{mini:N+} into account, we get
	\begin{align*}
		J_\lambda(u_\lambda )= \Theta_\lambda^+ \leq J_\lambda(\zeta(t)(u_\lambda +th)) = \psi_{u_\lambda +th}(\zeta(t))\leq \psi_{u_\lambda +th}(1)= J_\lambda(u_\lambda +th).
	\end{align*}
\end{proof}

Now we are ready to prove the existence of the first weak solution to problem \eqref{problem}.

\begin{proposition}\label{existence:first}
	Let hypotheses \textnormal{(H)} be satisfied and let $\lambda \in (0,\min\{\Lambda_1,\Lambda_2\})$, with $\Lambda_1$, $\Lambda_2$ given in Lemmas \ref{ksdp-lem3-} and \ref{ksdp-lem2}. Then, $u_\lambda$ is a weak solution of  problem \eqref{problem} with $J_\lambda(u_\lambda )<0$.
\end{proposition}

\begin{proof}
	We have to show that $u_\lambda >0$ a.\,e.\,in $\Omega$ and for every $\varphi \in \V$, $u_\lambda^{-\gamma} \varphi \in L^1(\Omega)$ and 
	\begin{align}\label{weak:form}
		\begin{split}
			m(\phi_\mathcal{H}(\nabla u_\lambda))\left\langle \mathcal{L}_{p,q}^{a}(u), \varphi\right\rangle
			= \lambda \int_{\Omega} u_\lambda^{-\gamma}\varphi\,\diff x +\int_\Omega u_\lambda^{r-1} \varphi\,\diff x.
		\end{split}	
	\end{align}
	We divide the proof into three steps.\\

	{\bf Step 1:} $u_\lambda >0$ \textit{a.\,e.\,in} $\Omega.$
	
	From Proposition \ref{mini:N+} we already know that $u_\lambda \geq 0$ a.\,e.\,in $\Omega$. In order to prove the strict positivity, we argue by contradiction. Suppose that there exists a set $K \subset \Omega$ with positive measure such that $u_\lambda \equiv 0$ in $K$.
	Applying  Proposition \ref{ksdp-prop4} with $h\in \V$ satisfying $h>0$ and let $t \in P_\delta\setminus\{0\}$, then $(u_\lambda + th)^{1-\gamma} > u_\lambda^{1-\gamma}$ in $\Omega \setminus K$, and we get
	\begin{align*}
		\begin{split}
			0 &\leq J_\lambda(u_\lambda + t h) - J_\lambda(u_\lambda) \\
			&= M[\phi_\mathcal{H}(\nabla (u_\lambda + th))] - M[\phi_\mathcal{H}(\nabla u_\lambda)] -\frac{\lambda}{1-\gamma}\int_\Omega \left[(u_\lambda + th)^{1-\gamma}- u_\lambda^{1-\gamma}\right]\,\diff x \\
			& \quad - \frac{1}{r} \left[\|u_\lambda + th\|_{r}^r - \|u_\lambda\|_r^r \right]\\
			& < M[\phi_\mathcal{H}(\nabla (u_\lambda + th))] - M[\phi_\mathcal{H}(\nabla u_\lambda)] -\frac{\lambda t^{1-\gamma}}{(1-\gamma) }\int_K h^{1-\gamma} \,\diff x 
			- \frac{1}{r} \left(\|u_\lambda + th\|_{r}^r - \|u_\lambda\|_r^r \right).
		\end{split}
	\end{align*}
	Dividing by $t>0$ and passing to the limit as $t \to 0^+$ in the estimate above, we conclude that
	\begin{align*}
		0 \leq \frac{J_\lambda(u_\lambda + t h) - J_\lambda(u_\lambda)}{t} \to - \infty,
	\end{align*}
	which is a contradiction. Thus, $u_\lambda >0$ a.\,e.\,in $\Omega$.\\
	
	{\bf Step 2:} \textit{For any $h \in \V$ with $h \geq 0$ let us verify}
	\begin{align}\label{weak:ineq}
		\begin{split}
			m(\phi_\mathcal{H}(\nabla u_\lambda))\left\langle \mathcal{L}_{p,q}^{a}(u), h\right\rangle
			\geq \lambda \int_{\Omega} u_\lambda^{-\gamma}h\,\diff x +\int_\Omega u_\lambda^{r-1} h\,\diff x.
		\end{split}
	\end{align}

	For this purpose, let us consider the nonnegative and measurable functions $\zeta_n \colon \Omega \to \mathbb{R}^+ $ defined by
	\begin{align*}
		\zeta_n(x) := \frac{(u_\lambda(x) + t_n h(x))^{1-\gamma} - u_\lambda(x)^{1-\gamma}}{t_n}
	\end{align*}
	where $\{t_n\}_{n \in \N}$ is a decreasing sequence such that $\lim_{n \to \infty} t_n =0$. Clearly, we have
	\begin{align*}
		\lim_{n \to \infty} \zeta_n (x)= (1-\gamma) u_\lambda(x)^{-\gamma} h(x) \quad \text{for a.\,a.\,} x \in \Omega.
	\end{align*}
	Now, by using Fatou's Lemma, we get
	\begin{align}\label{nuova}
		\lambda \int_{\Omega} u_\lambda^{-\gamma}h \,\diff x\leq \frac{\lambda}{1-\gamma} \liminf_{n \to \infty} \int_{\Omega} \zeta_n \,\diff x.
	\end{align}
	Arguing similarly to Step 1 and applying again Proposition \ref{ksdp-prop4}, we obtain
	\begin{align*}
		\begin{split}
			0 &\leq \frac{J_\lambda(u_\lambda + t_n h) - J_\lambda(u_\lambda)}{t_n} \\
			&= \frac{M[\phi_\mathcal{H}(\nabla (u_\lambda + t_nh))] - M[\phi_\mathcal{H}(\nabla u_\lambda)]}{t_n} -\frac{\lambda}{1-\gamma}\int_\Omega \zeta_n\,\diff x
			- \frac{1}{r}\int_\Omega \frac{(u_\lambda + t_n h)^r- u_\lambda^r}{t_n}\,\diff x.
		\end{split}
	\end{align*}
	Letting $n\to\infty$ in the inequality above and using \eqref{nuova} it follows that
	\begin{align*}
		\lambda \int_{\Omega} u_\lambda^{-\gamma}h \,\diff x
		\leq m(\phi_\mathcal{H}(\nabla u_\lambda))\left\langle \mathcal{L}_{p,q}^{a}(u), h\right\rangle- \int_\Omega u_\lambda^{r-1} h\,\diff x.
	\end{align*}
	Hence, \eqref{nuova} is satisfied and we also infer that $u_\lambda^{-\gamma} h \in L^1(\Omega)$. Therefore, we have $u_\lambda^{-\gamma} \varphi \in L^1(\Omega)$ for $\varphi\in\V$ since $\varphi=\varphi^+-\varphi^-$ with $\varphi^{\pm}=\max\{\pm \varphi,0\}$.\\

	{\bf Step 3:} $u_\lambda$ \textit{satisfies} \eqref{weak:form}.
	
	Let $\varphi\in \V$ and let $\varepsilon>0$. We take $h= (u_\lambda + \varepsilon \varphi)^+$ as test function in \eqref{weak:ineq} and use $u_\lambda \in \mathcal{N}_\lambda$ as well as $\Omega  = \{u_\lambda +\varepsilon \varphi > 0\}\cup \{u_\lambda +\varepsilon \varphi \leq 0\}$. This leads to
	\begin{align}\label{wealsol:est}
		\begin{split}
			0 &\leq m(\phi_\mathcal{H}(\nabla u_\lambda))\;\langle \mathcal{L}_{p,q}^a (u_\lambda),(u_\lambda + \varepsilon \varphi)^+ \rangle - \lambda \int_{\Omega} u_\lambda^{-\gamma} (u_\lambda + \varepsilon \varphi)^+ \,\diff x - \int_\Omega u_\lambda^{r-1} (u_\lambda + \varepsilon \varphi)^+ \,\diff x\\
			& = m(\phi_\mathcal{H}(\nabla u_\lambda))  \int_\Omega \left(|\nabla u_\lambda|^{p-2}\nabla u_\lambda+ a(x) |\nabla u_\lambda|^{q-2}\nabla u_\lambda \right) \cdot \nabla (u_\lambda + \varepsilon \varphi) \,\diff x \\
			& \quad - \lambda \int_{\Omega} u_\lambda^{-\gamma} (u_\lambda + \varepsilon \varphi) \,\diff x - \int_\Omega u_\lambda^{r-1} (u_\lambda + \varepsilon \varphi) \,\diff x\\
			&\quad - m(\phi_\mathcal{H}(\nabla u_\lambda))  \int_{\{u_\lambda + \varepsilon \varphi \leq 0\}} \left(|\nabla u_\lambda|^{p-2}\nabla u_\lambda+ a(x) |\nabla u_\lambda|^{q-2}\nabla u_\lambda \right)\cdot \nabla (u_\lambda + \varepsilon \varphi) \,\diff x \\
			&  \quad  + \lambda \int_{\{u_\lambda + \varepsilon \varphi \leq 0\}} u_\lambda^{-\gamma} (u_\lambda + \varepsilon \varphi) \,\diff x + \int_{\{u_\lambda + \varepsilon \varphi \leq 0\}} u_\lambda^{r-1} (u_\lambda + \varepsilon \varphi) \,\diff x\\
			&= \varepsilon \left[m(\phi_\mathcal{H}(\nabla u_\lambda))  \int_\Omega \left(|\nabla u_\lambda|^{p-2}\nabla u_\lambda+ a(x) |\nabla u_\lambda|^{q-2}\nabla u_\lambda \right) \cdot \nabla \varphi \,\diff x  - \int_{\Omega} \left[ \lambda u_\lambda^{-\gamma}  - u_\lambda^{r-1}\right] \varphi \,\diff x \right]\\
			&\quad  - m(\phi_\mathcal{H}(\nabla u_\lambda))  \int_{\{u_\lambda + \varepsilon \varphi \leq 0\}} \left(|\nabla u_\lambda|^{p-2}\nabla u_\lambda+ a(x) |\nabla u_\lambda|^{q-2}\nabla u_\lambda \right)\cdot \nabla (u_\lambda +\varepsilon \varphi) \,\diff x \\
			& \quad + \lambda \int_{\{u_\lambda + \varepsilon \varphi \leq 0\}} u_\lambda^{-\gamma} (u_\lambda + \varepsilon \varphi) \,\diff x + \int_{\{u_\lambda + \varepsilon \varphi \leq 0\}} u_\lambda^{r-1} (u_\lambda + \varepsilon \varphi) \,\diff x\\
			& \leq  \varepsilon \bigg[m(\phi_\mathcal{H}(\nabla u_\lambda))  \int_\Omega \left(|\nabla u_\lambda|^{p-2}\nabla u_\lambda+ a(x) |\nabla u_\lambda|^{q-2}\nabla u_\lambda \right)\cdot \nabla \varphi \,\diff x  - \int_{\Omega} \left[ \lambda u_\lambda^{-\gamma}  - u_\lambda^{r-1}\right] \varphi \,\diff x\\
			& \quad \quad - m(\phi_\mathcal{H}(\nabla u_\lambda))  \int_{\{u_\lambda + \varepsilon \varphi \leq 0\}} \left(|\nabla u_\lambda|^{p-2}\nabla u_\lambda+ a(x) |\nabla u_\lambda|^{q-2}\nabla u_\lambda \right) \cdot \nabla \varphi  \,\diff x\bigg].
		\end{split}
	\end{align}
	Since $|\{x \in \Omega\,:\, u_\lambda(x) + \varepsilon \varphi(x) \leq 0\}| \to 0$ as $\varepsilon \to 0$ by Step 1, we know that
	\begin{align}\label{conv:est}
		m(\phi_\mathcal{H}(\nabla u_\lambda))  \int_{\{u_\lambda + \varepsilon \varphi \leq 0\}} \left(|\nabla u_\lambda|^{p-2}\nabla u_\lambda+ a(x) |\nabla u_\lambda|^{q-2}\nabla u_\lambda \right)  \cdot \nabla \varphi  \,\diff x \to 0 \quad \text{as } \varepsilon \to 0.
	\end{align}
	Now, dividing \eqref{wealsol:est} by $\varepsilon$ and passing to the limit as $\varepsilon \to 0$ by using \eqref{conv:est}, we get
	\begin{align*}
		m(\phi_\mathcal{H}(\nabla u_\lambda))\left\langle \mathcal{L}_{p,q}^a (u_\lambda),\varphi \right\rangle \geq  \lambda \int_{\Omega} u_\lambda^{-\gamma} \varphi \,\diff x +\int_\Omega u_\lambda^{r-1} \varphi \,\diff x.
	\end{align*}
	The arbitrariness of $\varphi \in  \V$ implies that equality must hold.  Hence $u_\lambda\in\V$ is a weak solution of problem \eqref{problem} with $J_\lambda(u_\lambda)=\Theta_\lambda ^+ <0$. 
\end{proof}

Let us now prove the existence of a second weak solution of problem \eqref{problem}. For this, we minimize the energy functional $J_\lambda$ restricted to the set $\mathcal{N}_\lambda^-$. We define
\begin{align*}
	\Theta_\lambda ^- = \inf_{u\in\mathcal{N}_\lambda^-}J_\lambda(u).
\end{align*}

\begin{proposition}\label{pro:prelim:sec-sol}
	Let hypotheses \textnormal{(H)} be satisfied. Then there exists $\Lambda_3\in  (0,\min\{\Lambda_1,\Lambda_2\}]$,  with $\Lambda_1$, $\Lambda_2$ given in Lemmas \ref{ksdp-lem3-} and \ref{ksdp-lem2}, such that $\Theta_\lambda^- >0$ for all $\lambda \in (0, \Lambda_3)$. Moreover, for every $\lambda \in (0, \Lambda_3)$, there exists $v_\lambda \in \mathcal{N}_{\lambda}^-$ such that $v_\lambda \geq 0$ a.\,e.\,in $\Omega$ and $\Theta_\lambda^-= J_\lambda(v_\lambda)$.
\end{proposition}

\begin{proof}
	The first assertion will be proved by contradiction. Thus, let us suppose there exists $v_0 \in \mathcal{N}_\lambda^-$ such that $J_\lambda(v_0) \leq 0$, that is,
	\begin{align}\label{est:cont:stat}
		M[\phi_\mathcal{H}(\nabla v_0)] \leq \frac{\lambda}{1-\gamma}\int_\Omega |v_0|^{1-\gamma}\,\diff x + \frac{1}{r}\int_\Omega |v_0|^r\,\diff x.
	\end{align}
	Using $p \phi_{\mathcal{H}}(\nabla u) \geq \|\nabla u\|_p^p$ we obtain the following estimate
	\begin{align}\label{est:control}
		\begin{split}
			&M[\phi_\mathcal{H}(\nabla v_0)] - \frac{1}{q \theta} m(\phi_{\mathcal{H}}(\nabla v_0)) (\|\nabla v_0\|_p^p + \|\nabla v_0\|^q_{q,a})\\
			&= a_0\phi_\mathcal{H}(\nabla v_0) + \frac{b_0}{\theta} \phi^\theta_\mathcal{H}(\nabla v_0) - \frac{1}{q \theta} \left(a_0 + b_0 \phi^{\theta-1}_\mathcal{H}(\nabla v_0)\right) (\|\nabla v_0\|_p^p + \|\nabla v_0\|^q_{q,a})\\
			& = a_0 \left(\frac{1}{p} - \frac{1}{q \theta} \right)\|\nabla v_0\|_p^p + a_0 \left(\frac{1}{q} - \frac{1}{q \theta} \right)\|\nabla v_0\|^q_{q,a} + b_0 \phi^{\theta-1}_\mathcal{H}(\nabla v_0) \left(\frac{1}{p \theta} - \frac{1}{q \theta} \right)\|\nabla v_0\|_p^p \\
			& \geq \frac{1}{pq\theta}
			\frac{b_0 (q-p)}{p^{\theta-1}} \|\nabla v_0\|_p^{p\theta} \\
			& = D_3 \|\nabla v_0\|_p^{p\theta}\quad\text{with }D_3=\frac{1}{pq\theta}
			\frac{b_0 (q-p)}{p^{\theta-1}}>0.
		\end{split}
	\end{align}
	Since $v_0 \in \mathcal{N}_\lambda^- \subset \mathcal{N}_\lambda$ we have  
	\begin{align}\label{estimate-2}
		- \frac{1}{q \theta}m(\phi_{\mathcal{H}}(\nabla v_0))\left(\|\nabla v_0\|_p^p+ \|\nabla v_0\|_{q,a}^q\right)
		=- \frac{\lambda}{q \theta} \int_\Omega |v_0|^{1-\gamma}\,\diff x 
		- \frac{1}{q \theta} \int_\Omega |v_0|^r\,\diff x.
	\end{align}
	Now, using \eqref{est:cont:stat}, \eqref{est:control}, \eqref{estimate-2}, $r > q \theta$ along with H\"older's inequality and \eqref{best-sobolev-constant}, we obtain
	\begin{align*}
		D_3 \|\nabla v_0\|_p^{p\theta}
		&\leq M[\phi_\mathcal{H}(\nabla v_0)] - \frac{1}{q \theta} m(\phi_{\mathcal{H}}(\nabla v_0)) (\|\nabla v_0\|_p^p + \|\nabla v_0\|^q_{q,a}) \\
		& \leq \lambda \left(\frac{1}{1-\gamma}-\frac{1}{q\theta}\right) \int_{\Omega} |v_0|^{1-\gamma} \,\diff x + \left(\frac{1}{r}- \frac{1}{q \theta}\right) \int_{\Omega} |v_0|^{r} \,\diff x\\
		& \leq \lambda \left(\frac{q \theta + \gamma-1}{q \theta(1-\gamma)}\right) |\Omega|^{1-\frac{1-\gamma}{p^*}} S^{-\frac{1-\gamma}{p}}\|\nabla v_0\|_p^{1-\gamma}= \lambda D_4 \|\nabla v_0\|_p^{1-\gamma}
	\end{align*}
	with
	\begin{align*}
		D_4=\left(\frac{q \theta + \gamma-1}{q \theta(1-\gamma)}\right) |\Omega|^{1-\frac{1-\gamma}{p^*}} S^{-\frac{1-\gamma}{p}}>0.
	\end{align*}
	Combining the considerations above with Proposition \ref{gap-struc} gives
	\begin{align*}
		0<D_2^{\frac{p\theta-1+\gamma}{p}}\leq \|\nabla v_0\|_p^{p\theta-1+\gamma} \leq \frac{D_4}{D_3}\lambda.
	\end{align*}
	Letting $\lambda\to0$ yields a contradiction. Therefore, we can find $\Lambda_3\in  (0,\min\{\Lambda_1,\Lambda_2\}]$ such that $\Theta_\lambda^- >0$ for all $\lambda \in (0, \Lambda_3)$.

	Let us now prove the second assertion of the proposition. To this end, let $\{v_n\}_{n \in \N}$ be a minimizing sequence in $\mathcal{N}_{\lambda}^-$ such that $J_\lambda(v_n) \to \Theta_\lambda^-$. Since $\mathcal{N}_{\lambda}^- \subset \mathcal{N}_{\lambda}$, Lemma \ref{ksdp-lem1} implies that $\{v_n\}_{n\in\N}$ is a bounded sequence in $\V$. Therefore, by Proposition \ref{proposition_embeddings}\textnormal{(ii)} along with the reflexivity of $\V$, there exist a subsequence still denoted by $\{v_n\}_{n\in\N}$, and $v_\lambda\in\V$ such that
	\begin{align}\label{est:conv}
		v_n\rightharpoonup v_\lambda \quad\text{in }\V,
		\quad v_n\to v_\lambda\quad\text{in }\Lp{s}
		\quad\text{and}\quad v_n\to v_\lambda \quad \text{a.\,e.\,in }\Omega 
	\end{align}
	for any $s\in [1,p^*)$.
From $v_n \in \mathcal{N}_\lambda^-$, $p \phi_{\mathcal{H}}(\nabla u) \geq \|\nabla u\|_p^p$ and H\"older's inequality along with \eqref{best-sobolev-constant} we have
	\begin{align*}
		\frac{b_0 (p-1) S}{p^{\theta-1}|\Omega|^{\left(1- \frac{r}{p^*}\right)\frac{p}{r}}} \|v_n\|_r^p 
		&\leq  \frac{b_0 (p-1)}{p^{\theta-1}} \|\nabla v_n\|_p^p \\
		&\leq  \left[a_0 +b_0\phi_\mathcal{H}^{\theta-1}(\nabla v_n)\right] \left[(p-1)\|\nabla v_n\|_p^p+(q-1) \|\nabla v_n\|_{q,a}^q\right]\\
		& \quad \quad + b_0(\theta -1)\phi_\mathcal{H}^{\theta-2}(\nabla v_n)\left(\|\nabla v_n\|_p^p+\|\nabla u\|_{q,a}^q\right)^2+ \lambda \gamma \int_\Omega |v_n|^{1-\gamma}\,\diff x\\ 
		&\leq (r-1) \int_\Omega |v_n|^r\,\diff x= (r-1) \|v_n\|^r_r.
	\end{align*}
	This implies
	\begin{align*}
		0<\frac{b_0 (p-1) S}{p^{\theta-1}(r-1)|\Omega|^{\left(1- \frac{r}{p^*}\right)\frac{p}{r}}} \leq \|v_n\|_r^{r-p}.
	\end{align*}
	Hence, due to the strong convergence of $v_n\to v_\lambda$ in $\Lp{r}$, see \eqref{est:conv}, we can conclude that $v_\lambda\neq 0$. 
	
	Since $v_\lambda \not\equiv 0$, from Lemma \ref{ksdp-lem2}  we know there exists a unique $t_2^{v_\lambda}>0$ such that $t_2^{v_\lambda}v_\lambda\in \mathcal{N}_\lambda^-$.
	
	Now, we are going to show, up to a subsequence, that $\lim_{n \to \infty} \varrho(\nabla v_n)= \varrho(\nabla v_\lambda)$. For this, we repeat the same arguments as in the proof of Proposition \ref{mini:N+} by establishing
	\begin{align*}
		\liminf_{n\to  \infty} \|\nabla v_n\|_p^p= \|\nabla v_\lambda \|_p^p
		\quad \text{and}\quad 
		\liminf_{n\to  \infty} \|\nabla v_n\|_{q,a}^q= \|\nabla v_\lambda \|_{q,a}^q.
	\end{align*}

	Let us suppose that
	\begin{align*}
		\liminf_{n\to  \infty} \|\nabla v_n\|_p^p> \|\nabla v_\lambda \|_p^p.
	\end{align*}
	Then, using the inequality above, \eqref{est:conv} and the continuity and increasing property of the primitive of the Kirchhoff term $M$, we obtain
	\begin{align*}
		J_\lambda(t_2^{v_\lambda} v_\lambda) < \lim \inf_{n \to \infty} J_\lambda( t_2^{v_\lambda} v_n) \leq \lim_{n \to \infty} J_\lambda(t_2^{v_\lambda} v_n).
	\end{align*}
	Note that $v_n$ is the global maximum since $\psi''_{v_n}(1)<0$. Therefore, we have $J_\lambda(t_2^{v_\lambda} v_n) \leq J_\lambda(v_n)$ and since $t_2^{v_\lambda}v_\lambda \in \mathcal{N}_\lambda^-$, we conclude that
	\begin{align*}
		\Theta_\lambda^- \leq J_\lambda( t_2^{v_\lambda} v_\lambda) <  \lim_{n \to \infty} J_\lambda( v_n) = \Theta_\lambda^-.
	\end{align*}
Thus, we get a contradiction. The other case works similarly and so we know that there is a subsequence (not relabeled) such $\lim_{n \to \infty} \varrho(\nabla v_n)= \varrho(\nabla v_\lambda)$. Using the uniform convexity of the modular function $\rho_\mathcal{H}$ and the continuity of the energy functional $J_\lambda$, we get $v_n \to v_\lambda$ due to Proposition \ref{proposition_modular_properties}\textnormal{(v)} and $J_\lambda(v_n) \to J_\lambda(v_\lambda)= \Theta_\lambda^-$ up to a subsequence. 
	
	In order to show $v_\lambda \in \mathcal{N}_\lambda^-$, we use the fact that $ v_n \in \mathcal{N}_\lambda^-$ for every $n \in \mathbb{N}$ and so by passing to the limit in $\psi''_{v_n}(1) <0$, we obtain
	\begin{align*}
		& \left[a_0 +b_0\phi_\mathcal{H}^{\theta-1}(\nabla v_\lambda)\right] \left[(p-1)\|\nabla v_\lambda\|_p^p+(q-1) \|\nabla v_\lambda\|_{q,a}^q\right]\\ \nonumber
		& \quad + b_0(\theta -1)\phi_\mathcal{H}^{\theta-2}(\nabla v_\lambda)\left(\|\nabla v_\lambda\|_p^p+ \|\nabla v_\lambda\|_{q,a}^q\right)^2 + \lambda \gamma \int_\Omega |v_\lambda|^{1-\gamma}\,\diff x -(r-1)\int_\Omega |v_\lambda|^r\,\diff x \leq 0.
	\end{align*}
	Since $v_\lambda \not\equiv 0$ and using Lemma \ref{ksdp-lem3} with $\lambda \in (0,\min\{\Lambda_1,\Lambda_2,\Lambda_3\} )$, we infer that the equality cannot occur, so we have a strict inequality which gives $v_\lambda \in \mathcal{N}_\lambda^-$. As before, since we can work with $|v_\lambda|$ instead of $v_\lambda$, we can assume that $v_\lambda \geq 0$ a.\,e.\,in $\Omega$.
\end{proof}

Now we show that $v_\lambda$ obtained in Proposition \ref{pro:prelim:sec-sol} is indeed a weak solution of our problem \eqref{problem}.

\begin{proposition}\label{existence:second}
	Let hypotheses \textnormal{(H)} be satisfied and let $\lambda \in (0,\min\{\Lambda_1,\Lambda_2,\Lambda_3\} )$, with $\Lambda_1$, $\Lambda_2$, $\Lambda_3$ given in Lemmas \ref{ksdp-lem3-} and \ref{ksdp-lem2} as well as Proposition \ref{pro:prelim:sec-sol}. Then $v_\lambda$ is a weak solution of  problem \eqref{problem} with $J_\lambda(v_\lambda) >0$.
\end{proposition}

\begin{proof}
	The proof works similar to the one of Proposition \ref{existence:first}. Let $h \in\V$ with $h>0.$ As we already know, $v_\lambda \in \mathcal{N}_\lambda^-$ and $v_\lambda \geq 0$ a.\,e.\,in $\Omega$. Then from Lemma \ref{ksdp-lem3} and Proposition \ref{pro:prelim:sec-sol} there exist $\varepsilon>0$ and a continuous map $\zeta\colon B_\varepsilon(0) \to (0,\infty)$ such that $\zeta(th) \to 1$ as $t \to 0^+$ and
	\begin{align}\label{est:secsol-1}
		\begin{split}
			\zeta(th)(u+th) &\in \mathcal{N}_\lambda^- 
			\quad \text{and} \quad 
			\Theta_\lambda^- = J_\lambda(v_\lambda) \leq J_\lambda(\zeta(th)(v_\lambda +th)) 
			\quad \text{for all } t \text{ such that }  th \in B_\varepsilon(0).
		\end{split}
	\end{align}
	Now, we claim that $v_\lambda >0$ a.\,e.\,in $\Omega$. We proceed by contradiction. Suppose there exists a set $\widehat{K} \subset \Omega$ of positive measure such that $v_\lambda=0$ in $\widehat{K}$. Since $\zeta(th) \to 1$ as $t \to 0^+$, by continuity of the map $\psi_{v_\lambda}$ in the neighborhood of $1$ and $\psi_{v_\lambda}(1)$ being the global maximum, we can choose $\delta>0$ small enough such that $t \in [0, \delta]$ and $\psi_{v_\lambda}(1) \geq \psi_{v_\lambda}(\zeta(th)).$ Then, using \eqref{est:secsol-1}, one has
	\begin{align*}
		0 
		&\leq \frac{J_\lambda(\zeta(th)(v_\lambda +th)) - J_\lambda(v_\lambda)}{t}\\ 
		& \leq \frac{J_\lambda(\zeta(th)(v_\lambda +th)) - J_\lambda(\zeta(th) v_\lambda)}{t}\\
		& = \frac{M[\phi_\mathcal{H}(\nabla (\zeta(th)(v_\lambda +th)))] - M[\phi_\mathcal{H}(\nabla (\zeta(th) v_\lambda))]}{t} \\
		& \quad -\frac{\lambda}{1-\gamma}\int_\Omega \frac{(\zeta(th)(v_\lambda +th))^{1-\gamma}- (\zeta(th) v_\lambda )^{1-\gamma}}{t}\,\diff x \\
		& \quad - \frac{1}{r}\int_\Omega \frac{(\zeta(th)(v_\lambda +th))^r-  (\zeta(th) v_\lambda )^r}{t}\,\diff x\\
		& < \frac{M[\phi_\mathcal{H}(\nabla (\zeta(th)(v_\lambda +th)))] - M[\phi_\mathcal{H}(\nabla (\zeta(th) v_\lambda))]}{t} 
		-\frac{\lambda (\zeta(th))^{1-\gamma}}{1-\gamma}\int_{\widehat{K}} \frac{(th)^{1-\gamma}}{t}\,\diff x \\
		& \quad - \frac{1}{r}\int_\Omega \frac{(\zeta(th)(v_\lambda +th))^r-  (\zeta(th) v_\lambda )^r}{t}\,\diff x \to -\infty\quad \text{as }t\to 0^+,
	\end{align*}
	which is a contradiction and hence $v_\lambda>0$ a.\,e.\,in $\Omega$. The rest of the proof can be done by following the same arguments as in the proof of Proposition \ref{existence:first}, using $\psi_{v_\lambda}(1) \geq \psi_{v_\lambda}(\zeta(th))$ along with $v_\lambda >0$ and \eqref{est:secsol-1}.
\end{proof}

\begin{proof}[Proof of Theorem \ref{main_result}]
	Choosing $\lambda_*\in (0,\min\{\Lambda_1,\Lambda_2,\Lambda_3\} )$, the assertions of the theorem follow now from Propositions \ref{existence:first} and \ref{existence:second}.
\end{proof}

\section{A second Kirchhoff double phase problem}\label{sec final}

Inspired by \eqref{problem3} studied in \cite{Fiscella-Pinamonti-2020}, in this section we deal with a Kirchhoff problem of double phase type with $p$ and $q$ elliptic terms separated. Namely, we consider

\begin{equation}\label{problem_second}\tag{$\widetilde{P}_\lambda$}
	\begin{aligned}
		-m \left(\|\nabla u\|_p^p\right)\Delta_pu-m\left(\|\nabla u\|_{q,a}^q\right) \div\left(a(x)|\nabla u|^{q-2}\nabla u\right) &= \lambda u^{-\gamma} +u^{r-1} \quad&& \text{in } \Omega,\\
		u  &> 0 \quad && \text{in } \Omega,\\
		u  &= 0 &&\text{on } \partial\Omega,
	\end{aligned}
\end{equation}
satisfying the same structural assumption of \eqref{problem}. Problem \eqref{problem_second} has still a variational setting, where the corresponding energy functional $\widetilde{J}_\lambda\colon \V \to \R$ associated to problem \eqref{problem_second} is given by
\begin{align*}
	\widetilde{J}_\lambda (u) &= \frac{1}{p}M\left(\|\nabla u\|_p^p\right)+\frac{1}{q}M\left(\|\nabla u\|_{q,a}^q\right) -\frac{\lambda}{1-\gamma}\int_\Omega |u|^{1-\gamma}\,\diff x - \frac{1}{r}\int_\Omega |u|^r\,\diff x\\
	&=a_0\phi_{\mathcal H}(\nabla u)+\frac{b_0}{\theta}\left(\frac{\|\nabla u\|_p^{p\theta}}{p}+\frac{\|\nabla u\|_{q,a}^{q\theta}}{q}\right) -\frac{\lambda}{1-\gamma}\int_\Omega |u|^{1-\gamma}\,\diff x - \frac{1}{r}\int_\Omega |u|^r\,\diff x.
\end{align*}
Of course the critical points of $\widetilde{J}_\lambda$ coincides with the weak solutions of \eqref{problem_second}, verifying the following complete definition.

\begin{definition}
	A function $u \in \V$ is said to be a weak solution of problem \eqref{problem_second} if $u^{-\gamma}\varphi\in L^1(\Omega)$, $u>0$ a.e.\,in $\Omega$ and
	\begin{align*}
		&m\left(\|\nabla u\|_p^p\right)\int_{\Omega}|\nabla u|^{p-2}\nabla u\cdot\nabla\varphi\,\diff x
		+m\left(\|\nabla u\|_{q,a}^q\right)\int_{\Omega}a(x)|\nabla u|^{q-2}\nabla u\cdot\nabla\varphi\,\diff x\\
		&= \lambda \int_{\Omega} u^{-\gamma}\varphi\,\diff x 
		+\int_\Omega u^{r-1} \varphi\,\diff x
	\end{align*}
	is satisfied for all $\varphi \in \V$.
\end{definition}

Then, arguing similarly to Theorem \ref{main_result}, we are able to provide the following result.

\begin{theorem}\label{main_result2}
	Let hypotheses \textnormal{(H)}  be satisfied. Then there exists $\widetilde{\lambda}>0$ such that for all $\lambda \in (0,\widetilde{\lambda}]$ problem \eqref{problem_second} has at least two weak solutions $w_\lambda$, $z_\lambda \in \V$ such that $\widetilde{J}_\lambda(w_\lambda)<0<\widetilde{J}_\lambda(z_\lambda)$.
\end{theorem}

The proof of Theorem \ref{main_result2} works exactly as the one for Theorem \ref{main_result}, up to slight changes of constants.
For this, we omit the detailed proof and we just introduce the fibering function $\widetilde{\psi}_u\colon [0,\infty) \to \R$ defined for $u \in \V \setminus \{0\}$ by 
\begin{align*}
	\widetilde{\psi}_u(t)= \widetilde{J}_\lambda(tu)\quad \text{for all }t\geq 0,
\end{align*}
that is
\begin{align*}
	\widetilde{\psi}_u(t)
	= a_0\phi_\mathcal{H}(t\nabla u) 
	+\frac{b_0}{\theta}\left(\frac{t^{p\theta}\|\nabla u\|_p^{p\theta}}{p}+\frac{t^{q\theta}\|\nabla u\|_{q,a}^{q\theta}}{q}\right)
	- \lambda \frac{t^{1-\gamma}}{1-\gamma}\int_\Omega |u|^{1-\gamma}\,\diff x
	- \frac{t^r}{r}\int_\Omega |u|^r\,\diff x.
\end{align*}
In this case, we still have $\widetilde{\psi}_u \in C^{\infty}((0,\infty))$ satisfying for $t>0$
\begin{align*}
	\widetilde{\psi}_u'(t) 
	&= a_0\left(t^{p-1}\|\nabla u\|_p^p+t^{q-1} \|\nabla u\|_{q,a}^q\right)
	+b_0\left(t^{p\theta-1}\|\nabla u\|_p^{p\theta}+t^{q\theta-1}\|\nabla u\|_{q,a}^{q\theta}\right)\\
	&\quad- \lambda t^{-\gamma}\int_\Omega |u|^{1-\gamma}\,\diff x 
	- t^{r-1}\int_\Omega |u|^r\,\diff x
\end{align*}
while
\begin{align*}
	\widetilde{\psi}_u''(t)
	& =   a_0\left[(p-1)t^{p-2}\|\nabla u\|_p^p+(q-1)t^{q-2} \|\nabla u\|_{q,a}^q\right]\nonumber\\
	& \quad + b_0\left[(p\theta -1)t^{p\theta-2}\|\nabla u\|_p^{p\theta}+(q\theta -1)t^{q\theta-1} \|\nabla u\|_{q,a}^{q\theta}\right]\\
	&\quad  + \lambda \gamma t^{-\gamma-1}\int_\Omega |u|^{1-\gamma}\,\diff x 
	-(r-1) t^{r-2}\int_\Omega |u|^r\,\diff x.\nonumber
\end{align*}
From this, we can still set the Nehari manifold $\mathcal N_\lambda$ and the related submanifolds as done in Section \ref{sec_3}. Then, we construct the two solutions $w_\lambda$ and $z_\lambda$ of \eqref{problem_second} as minimizers of 
$$
\widetilde{\Theta}_\lambda ^- = \inf_{u\in\mathcal{N}_\lambda^-}\widetilde{J}_\lambda(u),
\qquad\widetilde{\Theta}_\lambda ^+ = \inf_{u\in\mathcal{N}_\lambda^+}\widetilde{J}_\lambda(u)
$$
completing the proof of Theorem \ref{main_result2}.

\section*{Acknowledgments}
A.\,Fiscella is member of the {Gruppo Nazionale per l'Analisi Ma\-tema\-tica, la Probabilit\`a e
le loro Applicazioni} (GNAMPA) of the {Istituto Nazionale di Alta Matematica ``G. Severi"} (INdAM).
A.\,Fiscella realized the manuscript within the auspices of the INdAM-GNAMPA project titled "Equazioni alle derivate parziali: problemi e modelli" (Prot\_20191219-143223-545) and of the FAPESP Thematic Project titled "Systems and partial differential equations" (2019/02512-5). R.\,Arora acknowledges the support of the Research Grant from Czech Science Foundation, project GJ19-14413Y.

\end{document}